\documentclass[a4paper,12pt]{article}
\usepackage{xcolor}
\usepackage[margin=1in]{geometry}
\usepackage[utf8]{inputenc}
\usepackage[english]{babel}
\usepackage{t1enc,lmodern}
\usepackage{amsmath,amssymb,amsthm,bm,,bbm,mathtools,xspace,booktabs,url}
\usepackage{graphicx,etoolbox}
\usepackage{hyperref}

\usepackage[shortlabels]{enumitem}

\setlist{
  listparindent=\parindent,
  parsep=0pt,
}

\usepackage[round]{natbib}

\usepackage{caption}
\usepackage[raggedright]{titlesec}
\usepackage{soul}
\numberwithin{equation}{section}

\theoremstyle{plain} 
\newtheorem{theorem}{Theorem}[section]
\newtheorem{lemma}[theorem]{Lemma}
\newtheorem{proposition}[theorem]{Proposition}
\newtheorem{definition}[theorem]{Definition}

\theoremstyle{definition} 

\newtheorem{remark}[theorem]{Remark}

\usepackage{authblk}

\makeatletter
\renewenvironment{abstract}{%
  \small%
  \providecommand\keywords{%
    \par\medskip\noindent\textit{Keywords:}\xspace}%
  \begin{center}%
    \bfseries \abstractname\vspace{-.5em}\vspace{\z@}%
  \end{center}%
  \quote%
}{\endquote}
\makeatother

\newcommand{\Vect}[1]{\mathbf{#1}}

\newcommand{\cG}{\mathcal{G}}

\newcommand{\cF}{\mathcal{F}}
\newcommand{\cS}{\mathcal{S}}

\newcommand{\R}{\mathbb{R}}

\newcommand{\N}{\mathbb{N}}

\newcommand{\sign}{\mbox{sign}}

\newcommand{\bX}{\mathbf{X}}
\newcommand{\bx}{\mathbf{x}}

\newcommand{\dd}{\mathrm{d}}

\newcommand{\dom}{\mathrm{Dom}}

\newcommand{\MSE}{\mathrm{MSE}}
\newcommand{\MLE}{\mbox{\scriptsize\sc{mle}}}
\newcommand{\PR}{\mbox{\scriptsize\sc{pr}}}
\newcommand{\Gain}{\mathrm{Gain}}

\def\norm#1{\|#1\|}
\DeclareMathOperator\E{\mathbb{E}}
\renewcommand{\P}{\mathbb{P}}

\def\1{\mathbbm{1}}

\newcommand{\com}[1]{#1}

\begin{document}

\title{Stein estimation of the intensity of a  spatial homogeneous Poisson point process}

\author[1]{Marianne Clausel}
\author[1]{Jean-Fran\c{c}ois Coeurjolly}
\author[1]{Jérôme Lelong}
\affil[1]{Univ. Grenoble Alpes, LJK, F-38000, Grenoble, France.\\

\texttt{Marianne.Clausel@imag.fr},
\texttt{Jean-Francois.Coeurjolly@upmf-grenoble.fr},\\
\texttt{Jerome.Lelong@imag.fr}.}

\date{}

\maketitle

\begin{abstract}
  In this paper, we revisit  the original ideas of Stein and propose an estimator of the intensity parameter of a homogeneous Poisson point process defined on $\R^d$ and observed on a bounded window. The procedure is based on a new integration by parts formula for Poisson point processes. We show that our Stein estimator outperforms the maximum likelihood estimator in terms of mean squared error. In many practical situations, we obtain a gain larger than 30\%.

\end{abstract}

\section{Introduction}

Spatial point processes are stochastic processes modeling points at random locations  in arbitrary domains. General references on this topic are \citet{daley08,stoyan:kendall:mecke:95,moeller:waagepetersen:04} who cover theoretical as well as practical aspects. Among all models, the reference is the Poisson point process, which models points without any interaction. When the Poisson point process has a probability measure invariant under translation, we say that it is stationary or homogeneous. In this paper, we consider a homogeneous Poisson point process defined on $\R^d$ and observed through a bounded window $W\subset\R^d$. This point process is characterized by the single intensity parameter $\theta>0$, which is the mean number of points per volume unit. It is well--known that the maximum likelihood estimator of the parameter $\theta$, defined as the ratio of the number of points lying in $W$ divided by its volume, is unbiased and efficient. In this work, we explain how to build superefficient and therefore biased estimators of $\theta$ by revisiting the original ideas of Stein.

Based on the pioneering works \cite{stein:56} and \cite{stein:61}, \cite{stein:81} explained how to design a whole collection of estimators for the mean $\mu$ of a $p$-dimensional Gaussian random vector $X$ by using the famous Stein formula for Normal random variables: for any differentiable function $g : \R^p \rightarrow \R$ such that $\E \|\nabla g(X)\| < +\infty$, the following integration by parts formula holds
\begin{align}
  \label{eq:gauss_ipp}
  \E (\nabla g(X) ) = \E( (X-\mu) g(X)).
\end{align}
Stein suggested to consider estimators of the form $X + \nabla \log f(X)$ for positive and sufficiently smooth functions $f$. For this class of estimators, he showed using~\eqref{eq:gauss_ipp} that the mean squared error is related to the expectation of $\nabla^2 \sqrt{f(X)} / \sqrt{f(X)}$ thus providing an easy construction of estimators achieving a mean squared error smaller than the one of the maximum likelihood estimator.

A close look at the methodology developed by Stein reveals the key role played by the integration by parts formula~\eqref{eq:gauss_ipp}, in which the involved differential operator is the classical notion of derivative. This remark proved to be of prime importance as these computations rely on the standard chain rule for the derivative operator related to the integration by parts formula. Hence, to extend this methodology to other frameworks, one first needs a derivative operator satisfying the classical chain rule and second an integration by parts formula for this operator. In the case of Gaussian processes, these objects are defined by the Malliavin calculus and Stein estimators have been proposed by~\cite{Privault2006607,privault2008}.

Let us focus on the Poisson case. An integration by parts formula already exists for functions of Poisson random variables. Let $Y$ be a Poisson random variable with parameter $\lambda$ and $f$ a sufficiently integrable real valued function, then it is known from~\cite{Chen75} that $\E(Y f(Y)) = \lambda \E(f(Y+1))$. However, this formula involves a discrete derivative operator which does not satisfy the standard chain rule and which, therefore, cannot be used as the basement for designing new estimators.

\com{Integration by parts formulae for Poisson processes have a long history, see~\cite{privault:09}, \cite{murr:2012} for a recent review. The differences are explained by the use of different concepts of differential operators. As already outlined, we ruled out results based on the finite difference operator since it does not satisfy the chain rule property. Two other classes of differential operators exist. The first one was developed by~\citet{albeverio:kondratiev:rockner:96} and was further investigated in different contexts, see \citet{albeverio:kondratiev:rockner:98,rockner:schied:99} or more recently~\citet{decreusefond:joulin:savy:10}.
The second class is based on the damped gradient, first introduced in the one-dimensional case by \citet{carlen:pardoux:90,elliott:tsoi:93} and further developed by \citet{fang:malliavin:93,prat:privault:99,privault:09,privault:torrisi:11}.
The main difference between these two classes is the space of Poisson functionals used to derive the integration by parts formula (see Section~\ref{sec:comp:referee:content} after our main result for more details). Note that links between these gradient operators exist, see e.g.~\citet{prat:privault:99}.
The key--ingredient to develop a Stein estimator is to obtain an integration by parts formula of the form $\E (\nabla F)= \E ( F (N(W)-\theta |W|))$ where $F$ is a Poisson functional, $\nabla$ is a gradient operator, $N(W)$ measures the number of points falling into a bounded domain $W$ of $\R^d$, $|W|=\int_W \dd u$ and $\theta$ is the intensity parameter of a homogeneous Poisson point process. Before 2009, none of the integration by parts formula available in the literature could be directly applied to get the required equation (see again Section~\ref{sec:comp:referee:content}). \citet{privault:reveillac:09} reworked the differential operator proposed by~\citet{carlen:pardoux:90} and managed to derive the desired equation in the one-dimensional case, but their differential operator could not be extended to spatial Poisson point processes. We aim at filling this gap in the present paper.
}

\com{In Section~\ref{s:malliavin:derivative}, we design a differential operator for functionals of a spatial homogeneous Poisson point process, which satisfies the classical chain rule and further leads to an integration by parts formula in Section~\ref{sec:ipp}. Sections~\ref{s:malliavin:derivative} and~\ref{sec:ipp}} heavily rely on the theory of closable and adjoint operators, which makes some of the proofs become technical. We have decided to gather all these technicalities in Appendix to avoid being diverted from our main objective, namely devising superefficient estimators on the Poisson space.
Based on this integration by parts formula and its related derivative operator, we propose in Section~\ref{sec:stein} a wide class of Stein estimators and study their mean squared errors. In Section~\ref{sec:sim}, we lead a detailed numerical study of our Stein estimator in several realistic examples. In particular, we explain how to pick in practice the estimator with minimum mean squared error within a given class and in the one--dimensional case, we compare it  to the estimator proposed by \cite{privault:reveillac:09}.

\section{Background and notation}\label{s:malliavin:derivative}

\subsection{Some notation}

Elements of $\R^d$ are encoded as column vectors, i.e. if $x \in \R^d$,
$x=(x_1,\dots,x_d)^\top$ and we denote their Euclidean norms by $\norm{x}^2 =
{x^\top x} = {\sum_{i=1}^d x_i^2}$.  Let $W$ be a bounded open set of
$\R^d$. For any $k \in \N$, the set $\mathcal{C}^k(W,\R^p)$ (resp.
$\mathcal{C}^k_c(W,\R^p)$) denotes the class of $k$-times continuously
differentiable functions defined on $W$ with values in $\R^p$ (resp. in a
compact subset of $\R^p$). Let $f: W \longrightarrow \R$ be a locally
integrable function. A function $h$ is said to
be the {weak derivative} on $W$ of $f$ w.r.t.  $x_i$ if for any
$\varphi\in\mathcal{C}^1_c(W,\R)$, we have
\[
  \int_{W} h(x)\varphi(x)\dd x=-\int_{W}f(x)\frac{\partial \varphi(x)}{\partial x_i} \dd x.
\]
When such a function $h$ exists, it is unique a.e. and we denote it by
${\partial f}/\partial{x_i}$ in the sequel.  When $d=1$, we  use the classical
notation $f'$ to denote the weak derivative of the function $f$. When all the
weak partial derivatives of a real--valued and locally integrable function $f$
defined on $W$ exist, we can define its {weak gradient} on $W$ as
\begin{equation}\label{e:grad}
 \nabla f(x)= \left(\frac{\partial f}{\partial x_1}(x),\dots,\frac{\partial
      f}{\partial x_d}(x)\right)^\top,\quad \forall \;x \in\R^d.
\end{equation}
For a locally integrable vector field $V=(V_1,\dots,V_d)^\top$ defined from
$W$ into $\R^d$ such that for all $i$, $V_i$ admits a weak partial derivative
w.r.t. to $x_i$, we define the {weak divergence} of $V$ on $W$ as
\begin{equation}\label{e:div:physicians}
  \nabla\cdot V(x)=\sum_{i=1}^d \frac{\partial V_i}{\partial x_i}(x),\,\quad
  \forall \; x\in\R^d.
\end{equation}

\subsection{Poisson point processes on $\R^d$}

For a countable subset $\bx$ of $\R^d$, we denote by $n(\bx)$ the number of elements in $\bx$. For any bounded Borel set $B$ of $\R^d$, $\bx_B$ stands for $\bx\cap B$ and $|B|$ stands for the Lebesgue measure of $B$. We define the set of locally finite configurations of points by $N_{lf} = \{
  \bx \subset \R^d : n(\bx_B)<\infty \; \mbox{ for all bounded sets } B\subset \R^d
  \}$.
We equip $N_{lf}$ with the $\sigma$-algebra $\mathcal N_{lf} = \sigma \big( \big\{
    \bx \in N_{lf} : n(\bx_B)=m  \big\}  : B\in \mathcal B_0, m\in \N\setminus\{0\}\big)$
  where $\mathcal B_0$ is the class of bounded Borel sets of $\R^d$. Then, a spatial point process $\bX$ on $\R^d$ is simply a measurable mapping on some probability space $(\Omega,\mathcal F,\P)$ with values in $(N_{lf},\mathcal N_{lf})$.

Let $W\subset\R^d$ be a compact set with positive Lebesgue measure $|W|$ playing the role of the observation window of $\bX$. \com{We assume that $W$ has a $\mathcal{C}^2$ boundary, so that the function $x\mapsto d(x,W^c)$ is also $\mathcal{C}^2$ in a neighborhood of $\partial W$.} We denote the number of points in $W$ by $N(W)=n(\bX_W)$; a realization of $\bX_W$ is
of the form $\bx=\{x_1,\dots,x_n\}\subset W$, for some $0\le
n<\infty$. If $n=0$, then $\bx=\emptyset$ is the empty point pattern in $W$. For
further background material and theoretical details on spatial point
process, see e.g.~\cite{daley:vere-jones:03,daley08} and
\cite{moeller:waagepetersen:04}. Given $N(W)=n$, we denote by $X_1,\dots,X_n\in W$ the $n$ location points.

In this paper, $\bX$ is a homogeneous Poisson point process, defined on $\R^d$, observed in $W$ and with intensity parameter $\theta>0$. Remember that the distribution of $\bX$  is entirely characterized by the void probabilities $\P(\bX\cap B=\emptyset) =e^{-\theta|B|}$ for any bounded $B\subset \R^d$. The more standard properties are: (i) $N(B)$ follows a Poisson distribution with parameter $\theta|B|$ for any bounded $B$. (ii) For $B_1,B_2,\dots$ disjoint sets of $\R^d$, $N(B_1),N(B_2),\dots$ are independent random variables. \com{Another characterization can be made using the generating function of $\bX$ (see e.g. \citet{moeller:waagepetersen:04}): for any  function $h:\R^d \to [0,1]$ setting $\exp(-\infty)=0$
\[
  \E  \prod_{u\in \bX} h(u)  = \exp \left( -\theta \int_{\R^d} (1-h(u))  \dd u \right).
\]
Let $\cF_W$ be the $\sigma-$field on $\Omega$ generated by the points of $\bX$ on $W$. In the following, we work on $(\Omega, \cF_W, \P)$ and write $L^2(\Omega) = L^2(\Omega, \cF_W, \P)$.}

\subsection{Poisson functionals and Malliavin derivative}
We introduce the following space
\begin{align} \label{eq:S}
  {\mathcal S} = \bigg\{ & F =f_0 \1 (N(W)=0) + \sum_{n\geq 1} \1(N(W)=n)
  f_n(X_1,\dots,X_n),
   \quad \mbox{ with } f_0 \in \R, \nonumber\\
  & \forall n \ge 1 \; f_n\in L^1(W^n,\R)
   \mbox{ is a symmetric function} \bigg\}.
\end{align}
The functions $f_n$ are called the {form functions} of $F$. Since $\bX$ is a Poisson point process, we have
\begin{equation} \label{eq:EF}
  \E[F]=e^{-\theta |W|}f_0 +e^{-\theta |W|} \sum_{n\geq 1} \frac{\theta^n}{n!} \int_{W^n}f_n(z_1,\dots,z_n)\dd z_1\dots \dd z_n.
\end{equation}
Note that with the choice of the $\sigma-$field ${\mathcal F}_W$, $L^2(\Omega) = \{F \in {\mathcal S} \; : \; \E[F^2] < \infty\}$.
For any $F \in L^2(\Omega)$, we denote the norm $\|\cdot\|_{L^2(\Omega)}$ by
\[
  \|F\|_{L^2(\Omega)}=\E[F^2]^{1/2} = \left(e^{-\theta |W|}f_0^2 +e^{-\theta |W|} \sum_{n\geq 1}
  \frac{\theta^n}{n!} \int_{W^n}f_n^2(z_1,\dots,z_n)\dd z_1\dots \dd
  z_n\right)^{1/2}.
\]
\com{In view of the expression of the norm, the
convergence in $L^2(\Omega)$ is linked to the convergence of the form functions.
\begin{lemma}\label{lem:cv}
  Let $\ell\geq 1$, $F,F_\ell\in L^2(\Omega)$ and $(f_{\ell,n})$ (resp $f_n$) be the form functions of the Poisson functionals $F_\ell$ (resp $F$).
We have
  $F_\ell\to F\mbox{ in }L^2(\Omega)$ iff
  \[
  e^{-\theta |W|}|f_{\ell,0}-f_0|^2 +e^{-\theta |W|} \sum_{n\geq 1}
  \frac{\theta^n}{n!} \int_{W^n}|f_{\ell,n}(z_1,\dots,z_n)-f_n(z_1,\dots,z_n)|^2\dd z_1\dots \dd
  z_n\to 0
  \]
  as $\ell\to\infty$.
\end{lemma}
The following subspace $\cS^\prime$ of $\cS$ plays a crucial role in the sequel
\begin{align*}
  \cS^\prime = \big\{F \in \cS: \,\exists C>0 \text{ s.t. } &\forall n\geq 1, f_n \in \mathcal C^1(W^n,\R)\text{ and }  \\
  & \|f_n\|_{L^\infty(W^n,\R)}+\sum_{i=1}^n\|\nabla_{x_i}f_n\|_{L^\infty(W^n,\R^d)}\leq C^n\big\}.
\end{align*}

}
\com{In particular, the definition of $\mathcal S^\prime$ ensures that $FG \in \mathcal S^\prime$ whenever $F,G\in \mathcal S^\prime$.} We fix a real--valued function $\pi:W^2 \to \R^d$, referred to as the weight function in the sequel. We assume that $\pi$ is bounded and that for a.e. $x\in W$, $z\mapsto \pi(z,x)$ belongs to $\mathcal{C}^1(W,\R^d)$. For any $x \in W$, we denote by $D_x^\pi$ the following differential operator defined for any $F\in \mathcal S^\prime$ by
\begin{equation} \label{eq:def1}
D_x^{\pi} F =-\sum_{n\geq 1} \1(N(W)=n) \sum_{i=1}^n (\nabla_{x_i} f_n)(X_1,\dots,X_n) \pi(X_i,x)
\end{equation}
where $\nabla_{x_i}f_n$ stands for the gradient vector of  $x_i\mapsto f_n(x_1,\dots,x_{i-1},x_i,x_{i+1},\dots,x_n)$. The operator $D^\pi$ is a Malliavin derivative operator satisfies the classical differentiation rules.
\begin{lemma}\label{lem:productgf}
  Let $F,G\in\mathcal{S}^\prime$ and $g\in\mathcal{C}^1_b(\R, \R)$. Then $FG\in \mathcal{S}^\prime$ and $g(F)\in \mathcal S^\prime$ and for any $x\in W$
\[
D_x^{\pi} (FG)=(D_x^{\pi} F)G+F(D_x^{\pi} G)
\quad \mbox{ and } \quad
D_x^{\pi} g(F)=g'(F)D_x^{\pi} F.
\]
\end{lemma}
To have such rules, we had to consider a bespoke differential operator, which differs from the standard Malliavin derivative on the Poisson space (see e.g. \cite{privault:94}). Before establishing an integration by parts formula, we define the subset $\dom(D^\pi)$ of $\mathcal S'$ as
\begin{align}
\dom(D^\pi) =& \bigg\{ F \in \mathcal S^\prime: \forall n \geq 1 \mbox{ and }  z_1,\dots,z_n \in \R^d
& \nonumber\\
& {f_{n+1}}_{\big| z_{n+1}\in \partial W}(z_1,\dots,z_{n+1}) = f_n(z_1,\dots,z_n),\,
 {f_1}_{\big| z\in \partial W}(z)=f_0 \bigg\}. \label{eq:compatibility}
\end{align}
\com{The notation $D^\pi F$ stands for} the random field $x\in W\mapsto D_x^\pi F$.
The operator $D^\pi$ is defined from $\dom(D^\pi)\subset \mathcal
S^\prime$ on $L^2(\Omega,L^2(W,\R))$,
where $L^2(\Omega,L^2(W,\R))$ is the space of random fields $Y$ defined on $W$ such that
$\|Y\|_{L^2(\Omega,L^2(W,\R))}=\left(\E\left[\int_{W}|Y(x)|^2\dd x\right]\right)^{1/2}<\infty$.

\com{The link between $\dom(D^\pi)$ and $L^2(\Omega)$ is presented in the next result and proved in Appendix~\ref{s:density-dom}.
\begin{lemma}\label{lem:density-dom}
The set $\dom(D^\pi)$ defined by~\eqref{eq:compatibility} is dense in $L^2(\Omega)$.
\end{lemma}

}

\section{Integration by parts formula}
\label{sec:ipp}

\subsection{Duality formula} \label{sec:duality}

In this section, we aim at extending the Malliavin derivative $D^\pi$ to a
larger class of Poisson functionals  by using
density arguments. We also prove an integration by parts formula for the
Malliavin derivative, involving the extension $\overline{D}^\pi$ of the
operator $D^{\pi}$ and its adjoint. We
start with basic definitions of closable operators --- i.e. operators which
can be extended by density --- and of the adjoint of a densely defined
operator.

\begin{definition}
Let $H_1,H_2$ be two Hilbert spaces and $T$ be a linear operator defined from
$\dom(T)\subset H_1\rightarrow H_2$. The operator $T$ is said to be closable if
and only if for any sequence $(f_\ell)\subset \dom(T)$ such that
$f_\ell\rightarrow 0$ (in $H_1$) and $Tf_\ell\rightarrow g$ (in $H_2$) then
$g=0$.
\end{definition}
The main point is that any closable operator $T$ from $\dom(T)\subset H_1\rightarrow H_2$ can be extended by density. Set
$\dom(\overline{T})=\big\{f\in H_1,\,\exists (f_\ell)\in \dom(T),\,f_\ell\to f\mbox{ in }H_1\mbox{ and }(T f_\ell) $ converges in $H_2\big\}$
and define for any $f=\lim f_\ell\in \dom(\overline{T})$ with $(f_\ell)\in \dom(T)$, $
\overline{T}f=\lim_{\ell\to\infty} Tf_\ell$.
Then, the operator $\overline{T}$ is called the {closure} of $T$. By the closability of $T$, the above limit does not depend on the chosen sequence $(f_\ell)$.

If we are given a linear operator $T$ such that $\overline{\dom(T)}=H_1$, an
element $g$ of $H_2$ is said to belong to $\dom(T^*)$, where $T^*$ is the
adjoint of $T$, if for any $f\in \dom(T)$, there exists some $h\in H_1$ such
that $<Tf,g>_{H_2}=<f,h>_{H_1}$. When such an $h$ exists, the assumption
$\overline{\dom(T)}=H_1$ ensures it is unique and given by $h=T^*g$. The following
{duality relation} between $T$ and $T^*$ holds:
\begin{equation}\label{e:duality}
<Tf,g>_{H_2}=<f,T^*g>_{H_1}\,\forall (f,g)\in \dom(T)\times \dom(T^*).
\end{equation}
In the case of the Malliavin derivative, this duality relation leads to an integration by parts formula.

\begin{theorem}[Duality relation]  \label{thm:ipp}
The operator $D^\pi$ is closable and admits a
closable adjoint $\delta^\pi$ from $L^2(\Omega,L^2(W,\R^d))$ into $L^2(\Omega)$
and the following duality relation holds:
\begin{equation}\label{eq:ipp}
\E \left[ \int_{W} D_x^\pi F \cdot V(x)\dd x  \right] = \E \left[ F
  \delta^\pi(V)\right], \; \forall F\in \dom(D^\pi),\, \forall V\in \dom(\delta^\pi).
\end{equation}
\com{In particular, \eqref{eq:ipp} extends to the case $F \in \dom(\overline{D^\pi})$, $V \in \dom(\overline{\delta}^\pi)$.  Let $V\in L^\infty(W,\R)$, we define  $\mathcal V^\pi: W \rightarrow \R^d$ by
\begin{equation}\label{eq:Vtilde}
  \mathcal V^\pi(u)=\int_{W} V(x)\pi(u,x)\dd x
\end{equation}
which is an element of $\mathcal{C}^1(W,\R^d)$.} We have the following explicit expression for $\delta^\pi$:
\begin{equation}\label{eq:delta:pi}
\delta^\pi(V)= \sum_{u\in \bX_W} \nabla\cdot \com{\mathcal V^\pi}(u)-\theta \int_W \nabla\cdot \com{\mathcal V^\pi(u)} \dd u.
\end{equation}
\end{theorem}
\noindent The proof of this result shares some similarities with \citet[Proposition 4.1]{privault:torrisi:11} and relies on the same tool, namely the standard trace Theorem \com{(see e.g. \citet{evans:gariepy:91})}, recalled
hereafter.
\begin{theorem}
  \label{th:trace}
  Let $B$ be a bounded subset of $\R^d$ with Lipschitz boundary and closure $W=\overline{B}$. Let $\mathcal{V} =(\mathcal{V}_1,\dots,\mathcal{V}_d)^T\in \mathcal{C}^1(R^d,\R^d)$ be a vector field and $g \in \mathcal{C}^1(\R^d,\R)$ be a real--valued function. Then,
  \begin{equation}\label{e:IPP1}
    \int_{W} (\nabla g)(x) \cdot \mathcal{V}(x) \dd x = -
    \int_{W} g(x) (\nabla\cdot \mathcal{V})(x) \dd x + \int_{\partial
      B} g(x) \mathcal{V}(x) \cdot \nu(\dd x)
  \end{equation}
  where $\nu$ stands for the outer normal to $\partial B$. When $g\equiv 1$, we get
  \begin{equation}\label{e:IPP2}
    \int_W \nabla\cdot \mathcal{V}(x) \dd x =
    \int_{\partial W}  \mathcal{V}(x) \cdot \nu(\dd x).
  \end{equation}
\end{theorem}

\begin{proof}[Proof of Theorem~\ref{thm:ipp}]
\noindent{\it Step 1: weak duality relation}. Assume that
\com{$F\in \dom(D^\pi)$} and $V\in L^\infty(W,\R^d)$. Let us prove
that~\eqref{eq:ipp} holds. Using standard results on Poisson processes, which are
in particular justified by the fact that \com{$F\in \dom(D^\pi)$}, we get
\begin{align*}
&\E \left[ \int_{W} D_x^\pi F \cdot V(x)\dd x  \right]\\
&= - e^{-\theta |W|} \sum_{n\geq 1} \frac{\theta^n}{n!} \sum_{i=1}^n \int_{W^n} \left(\int_{W} \nabla_{z_i}f_n(z_1,\dots,z_n) \cdot V(x) \; \pi(z_i,x)\dd x	\right) \dd z_1\dots \dd z_n \\
&=- e^{-\theta |W|} \sum_{n\geq 1} \frac{\theta^n}{n!} \sum_{i=1}^n \int_{W^{n-1}} \dd z_1 \dots \dd z_{i-1} \dd z_{i+1} \dots \dd z_n \int_W
\nabla_{z_i} f_n(z_1,\dots,z_n) \cdot \com{\mathcal V^\pi}(z_i) \dd z_i.
\end{align*} 	
Since $\pi,V$ are both bounded on $W$ and since for a.e. $x\in W$, $z\mapsto \pi(z,x)$ belongs to $\mathcal{C}^1(W,\R)$, then $\com{\mathcal{V}^\pi}\in \mathcal{C}^1(W,\R)$. Hence, we can apply Theorem~\ref{th:trace}. Using the compatibility conditions~\eqref{eq:compatibility}, we deduce that for $i=1,\dots,n$
\begin{align*}
 -\int_W \nabla_{z_i} &f_n(z_1,\dots,z_n) \cdot \com{\mathcal V^\pi}(z_i) \dd z_i	\\
&=  \int_W f_n(z_1,\dots,z_n) \nabla\cdot \com{\mathcal V^\pi}(z_i) \dd z_i
- \int_{\partial W} f_n(z_1,\dots,z_n) \com{\mathcal V^\pi}(z_i) \dd \Vect \nu_{z_i} \\
&=\int_W f_n(z_1,\dots,z_n) \nabla\cdot \com{\mathcal V^\pi}(z_i) \dd z_i	
-  f_{n-1}(z_1,\dots,z_{n-1}) \int_{ W} \nabla\cdot \com{\mathcal V^\pi}(u) \dd u.
\end{align*}
The last equation comes from~\eqref{e:IPP2} and the symmetry of the functions $f_n$. Therefore,
\begin{align*}
  \E \bigg[ \int_{W}& D_x^\pi  F \cdot V(x)\dd x  \bigg]\\
  =& e^{-\theta|W|} \sum_{n\geq 1} \frac{\theta^n}{n!} \times \sum_{i=1}^n \int_{W^n}f_n(z_1,\dots,z_n)\nabla\cdot \com{\mathcal V^\pi}(z_i)\dd z_1\dots \dd z_n \\
  &- e^{-\theta|W|} \sum_{n\geq 1} \frac{\theta^n}{n!} \times n \int_{W^{n-1}}f_{n-1}(z_1,\dots,z_{n-1}) \bigg(\int_W \nabla\cdot \com{\mathcal V^\pi}(u)\dd u\bigg) \dd z_1\dots \dd z_{n-1} \\
  =& \E[ F \sum_{u\in \bX_W} \nabla\cdot \com{\mathcal V^\pi}(u)]\\
  &- \theta e^{-\theta|W|}
  \bigg(\int_{W}\nabla\cdot \com{\mathcal V^\pi}(u)\dd u\bigg) \sum_{n\geq 1}
  \frac{\theta^{n-1}}{(n-1)!} \int_{W^{n-1}}\!\!\!\!f_{n-1}(z_1,\dots,z_{n-1})\dd
  z_1\dots \dd z_{n-1}.
\end{align*}
The last equality ensues from the invariance of the functions $f_n$ and the
stability of the domain $W^{n-1}$ by exchanging the coordinates. Then, we
deduce the result.\\

\noindent{\it Step 2: Extension of $\delta^\pi$ on a dense subset of
  $L^2(\Omega,L^2(W,\R^d))$. Validity of~\eqref{eq:ipp} on this dense subset.} Remember that $L^2(\Omega,L^2(W,\R^d))=\overline{L^2(\Omega)\otimes L^2(W,\R^d)}$.
Since by Lemma~\ref{lem:density-dom} $\dom(D^\pi)$ is a dense subset of $L^2(\Omega)$ and $L^\infty(W,\R^d)$ is a dense subset of $L^2(W,\R^d)$, we deduce that
$L^2(\Omega,L^2(W,\R^d))=\overline{\dom(D^\pi)\otimes L^\infty(W,\R^d)}$. Now, we
extend the operator $\delta^\pi$ on $\dom(D^\pi)\otimes L^\infty(W,\R^d)$ and then
prove~\eqref{eq:ipp} on this dense subset of $L^2(\Omega,L^2(W,\R^d))$. To this
end, we consider $G\in\dom(D^\pi)$, $V\in L^\infty(W,\R^d)$ and set
\[
\delta^\pi(GV)=G\delta^\pi(V)-\int_{\R^d}D_x^\pi G \cdot V(x)\dd x.
\]
Using the product rule, which is valid for any $F,G\in\dom(D^\pi)$, we deduce that
\begin{align*}
\E \left[ G\int_{W} D_x^\pi F \cdot V(x)\dd x  \right]=&\E \left[ \int_{W} D_x^\pi (FG) \cdot V(x)\dd x  -F\int_{W} D_x^\pi G \cdot V(x)\dd x  \right]\\
=&\E \left[FG\delta^\pi(V)-F\int_{W} D_x^\pi G \cdot V(x)\dd x  \right]\\
=&\E \left[F\delta^\pi(GV)\right].
\end{align*}
The second equality comes from the duality relation~\eqref{eq:ipp} applied to
$FG$ as an element of $\dom(D^\pi)$ and $V$ as an element of
$L^\infty(W,\R^d)$, whereas the last equality comes from the
definition of our extension of $\delta^\pi$ to $\dom(D^\pi)\otimes
L^\infty(W,\R^d)$.\\

\noindent{\it Step 3: closability of the operator $D^\pi$ and extension of~\eqref{eq:ipp} to $L^2(\Omega,L^2(W,\R^d))$}.
We extend~\eqref{eq:ipp} from $\dom(D^\pi)\otimes L^\infty(W,\R^d)$ to
$L^2(\Omega,L^2(W,\R^d))$ by proving that the operator $D^\pi$ is closable. Since $\dom(D^\pi)\otimes L^\infty(W,\R^d)$ is dense in $L^2(\Omega,L^2(W,\R^d))$, Theorem~\ref{th:adjoint2} justifies the extension of the duality relation~\eqref{eq:ipp} to $\dom(\overline{D}^\pi)\times \dom(\overline{\delta}^\pi)$ as stated in Theorem~\ref{thm:ipp}.

To prove that $D^\pi$ is closable, we  consider a sequence of elements $(F_\ell)\in \dom(D^\pi)$
such that $F_\ell\rightarrow 0$ in $L^2(\Omega)$. Assume also that $D^\pi
F_\ell\rightarrow U$ for some $U$ in \linebreak$L^2(\Omega,L^2(W,\R^d))$. We need to prove that
$U=0$, which is done using the following computations for any $G\in\dom(D^\pi)$ and any $V\in
L^\infty(W,\R^d)$
\begin{align*}
\bigg|\E\bigg[&\int_{W}U(x)GV(x)\dd x\bigg]\bigg|\\
&\leq \left|\E\left[F_\ell\delta^\pi(GV)\right]-\E\left[\int_{\R^d}U(x)GV(x)\dd x\right]\right|+\left|\E\left[F_\ell\delta^\pi(GV)\right]\right|\\
&\leq \left|\E\left[G\int_{W}D^\pi_x F_\ell \cdot V(x)\dd x\right]-\E\left[G\int_{W}U(x)V(x)\dd x\right]\right|+\left|\E\left[F_\ell\delta^\pi(GV)\right]\right|\\
&\leq \|G\|_{L^2(\Omega)}\big\|\int_{W}(D^\pi_x F_\ell-U(x)) \cdot V(x)\dd x\big\|_{L^2(\Omega)}+\|F_\ell\|_{L^2(\Omega)}\|\delta^\pi(GV)\|_{L^2(\Omega)}\\
&\leq \|G\|_{L^2(\Omega)}\|V\|_{L^2(W,\R^d)}\|D^\pi_x F_\ell-U(x)\|_{L^2(\Omega,L^2(W,\R^d))}+\|F_\ell\|_{L^2(\Omega)}\|\delta^\pi(GV)\|_{L^2(\Omega)}.
\end{align*}
The conclusion ensues from Theorem 7 of Chapter~2
of~\cite{birman:solomjak:87}, which is rephrased in Theorem~\ref{th:adjoint2} for the sake of completeness. Equation~\eqref{eq:ipp} is recovered by applying Theorem~\ref{th:adjoint2}  with $T^*=\delta^\pi$, $T=D^\pi$, $H_1=L^2(\Omega,L^2(W,\R))$ and $H_2=L^2(\Omega)$.
\end{proof}

\begin{theorem}\label{th:adjoint2}
Let $H_1,H_2$ be two Hilbert spaces and $T$ be a linear operator defined from $\dom(T)\subset H_1\rightarrow H_2$. Assume that $\overline{\dom(T)}=H_1$.  Then, $\overline{\dom(T^*)}=H_2$ if and only if $T$ is closable. In this case, $T^{**}$ exists \com{and}  coincides with $\overline{T}$. Then, (\ref{e:duality}) can be extended as follows:
$<f,T^*g>_{H_1}=<\overline{T}f,g>_{H_2}$, $\forall (f,g)\in \dom(\overline{T})\times\dom(T^*)$.
\end{theorem}

For any $F \in \dom(D^\pi)$ and any $V \in L^2(\Omega,L^2(W,\R^d))$, we define the operator $\nabla^{\pi,V}: \dom(D^\pi) \to \R$ by
\begin{equation}\label{e:nabla:simplified}
\nabla^{\pi,V} F= \int_{W} D_x^\pi F \cdot V(x) \dd x=-\sum_{n\geq 1}\1(N(W)=n)\sum_{i=1}^n \nabla_{x_i}f_n(X_1,\dots,X_n)\cdot \com{\mathcal{V}^\pi}(X_i).
\end{equation}
Note that the closability of $D^\pi$ implies the one of  the operator
$\nabla^{\pi,V}$.
\com{For the reader's convenience, we will use the same notation to denote the operators $\nabla^{\pi,V}$, $\delta^\pi$ and their closures.}

\subsection{A specific choice for $\pi$ and $V$}

In this section, we focus on a specific choice of functions $\pi$ and $V$ leading to a specific definition of the gradient of a Poisson functional. This choice is guided by two important remarks. First, Lemma~\ref{lem:specific:case} will underline that the key-ingredient to derive a superefficient estimator is to define a gradient $\nabla^{\pi,V}$ such that
\begin{equation}\label{eq:ipp2}
\E [\nabla^{\pi,V} F] = \E \left[ F(N(W) -\theta |W|)\right].
\end{equation}
From \eqref{e:nabla:simplified} and Theorem~\ref{thm:ipp}, this is achieved if $\nabla\cdot \com{\mathcal V^\pi} \equiv 1$. Second, to agree with the isotropic nature of the homogeneous Poisson point process, it is natural to define a Stein estimator being also isotropic. As pointed out by Proposition~\ref{prop:MSEStein}, this can be achieved by defining a Malliavin derivative which transforms an isotropic Poisson functional into an isotropic Poisson functional at any point. We do not want to go into these details right now but Lemma~\ref{lem:deriv:closed} suggests to require that both $\pi$ and $V$ be isotropic. Now, we \com{detail} a pair of functions $(\pi,V)$ satisfying the above requirements.

\begin{proposition}\label{lem:specific:case}
  \com{Let $(V_m)_{1 \le m \le d}$ be an orthonormal family of bounded functions of $L^2(W, \R)$.}
For any $x,y \in \R^d$, set $V(x)=d^{-1/2}(V_1(x),\dots,V_d(x))^\top$ and $\pi(y,x)=y^\top V(x)$.
Then, $\com{\mathcal V^\pi}(y)={y}/{d}$, which implies that $\nabla \cdot\com{\mathcal V^\pi}\equiv 1$.
With the above choices, we simply denote $\com{ \mathcal V=\mathcal V^\pi}$, $D=D^\pi$, $\nabla=\nabla^{\pi,V}$. From~\eqref{e:nabla:simplified}, we deduce that
\begin{equation} \label{e:nabla}
\nabla F = -\frac{1}{d}\sum_{n\geq 1} \1(N(W)=n) \sum_{i=1}^n \nabla_{x_i} f_n(X_1,\dots,X_n) \cdot X_i.
\end{equation}
Finally, for any $F\in\dom(\overline{D})$, \eqref{eq:ipp2} holds.
\end{proposition}

\begin{proof}
By definition, we have for any $y=(y_1,\dots,y_d)^\top\in\R^d$ and any $m\in\{1,\dots,d\}$
\begin{align*}
\mathcal V_m(y)&=\int_{\R^d}V_m(x)\pi(y,x)\dd x=d^{-1}\int_W \left(\sum_{m'} y_{m'} V_{m'}(x)\right)V_m(x) \dd x\\
&=d^{-1}\int_W\left(\sum_{m'} y_{m'} V_{m'}(x)V_m(x)\right)\dd x=d^{-1}\int_W y_m V_m^2(x)\dd x=y_m/d.
\end{align*}
Plugging this result in Equation~\eqref{e:nabla:simplified} yields~\eqref{e:nabla}.
Equation~\eqref{eq:ipp2} can be extended to $\dom(\overline D)$ using the closability of
the operator $\nabla$, which follows from the one of
$\nabla^{\pi,V}$ in the general case.
\end{proof}

Again, we want to stress the fact that other choices of pairs of functions
$(\pi,V)$ may lead to~\eqref{eq:ipp2} like  the simple choice
$V(x)=(d|W|)^{-1/2}  (1,\dots,1)^\top \1(x\in W)$ and $\pi(y,x)=y^\top V(x)$. However, the gradient derived from this
choice would not preserve the isotropy of an isotropic Poisson functional
anymore and would lead to technical difficulties especially in the formulation
of~\eqref{e:def:nablaF} in Lemma~\ref{lem:deriv:closed}, which should take into
account the jumps induced by the discontinuity of the form functions.

\subsection{Comparison with alternative versions of the Malliavin derivative} \label{sec:comp:referee:content}
\com{To finish this section, we give some insights into an alternative version of the Malliavin derivative also leading to an integration by parts formula of the form~\eqref{eq:ipp} but with unfortunately more restrictive assumptions on the possible functions $\mathcal{V}$. We refer to~\cite{albeverio:kondratiev:rockner:96,albeverio:kondratiev:rockner:98} and \citet[Section~8]{prat:privault:99} for more details on what follows. We briefly summarize their approach for a Poisson point process lying in $\R^d$. The authors consider the class of smooth cylindrical Poisson functionals of the form
\[
F=f\left(\sum_{u\in \bX_W}\varphi_1(u),\cdots,\sum_{u\in \bX_W}\varphi_p(u)\right)
\]
where $p$ is an integer, $f$ is an infinitely differentiable and bounded function on $W$ and $\varphi_1,\cdots,\varphi_p$ are infinitely differentiable on $W$, all compactly supported with $\mathrm{supp}(\varphi_i)\subset \overset{\circ}{W}$ for any $i=1,\cdots,p$.

Then, for any $u\in W$, the Malliavin derivative of $F$ at $u$ is defined by
\[
\widetilde{D}_u F=\sum_{i=1}^p \partial_i f\left(\sum_{v\in \bX_W}\varphi_1(v),\cdots,\sum_{v\in \bX_W}\varphi_p(v)\right)\nabla \varphi_i(u)
\]
Let $\mathcal{V} : W \to \R^d$ be an infinitely differentiable function  with $\mathrm{supp}(\mathcal{V})\subset \overset{\circ}{W}$. Then, the following formula holds (see  \citet[equation (8.5.6)]{prat:privault:99})
\begin{align}
  \label{eq:ipp-albeverio}
  \E\left[\widetilde{\nabla} F\right]= \E\left[F\left(\sum_{u\in \bX_W} \nabla\cdot \com{\mathcal V}(u)-\theta \int_W \nabla\cdot \com{\mathcal V(u)} \dd u\right)\right]
\end{align}
with $\widetilde{\nabla} F=\sum_{u\in \bX_W}\widetilde{D}_u F\cdot\mathcal{V}(u)$.

The two integration by parts formulae~\eqref{eq:ipp} and~\eqref{eq:ipp-albeverio} look very similar although the gradient operators are completely different. However, the constraint on ${\mathcal V}$ to obtain \eqref{eq:ipp-albeverio} prevents us from taking $\mathcal V$ such that $\nabla\cdot \com{\mathcal V}\equiv 1$ on $W$, which is crucial to get \eqref{eq:ipp2}.}

\section{Stein estimator} \label{sec:stein}

\subsection{Main results}\label{sec:stein:results}

\com{The maximum likelihood estimator of the intensity $\theta$ of the
spatial Poisson point process $\bX$ observed on $W$ is given by $\widehat \theta_{\MLE}=N(W)/|W|$ (see e.g. \citet{moeller:waagepetersen:04}). In this section, we propose a Stein estimator
derived from  the maximum likelihood estimator of the form}
\begin{equation}\label{eq:defEst}
\widehat \theta = \widehat \theta_{\MLE} + \frac1{|W|} \zeta = \frac1{|W|}(N(W)+\zeta)
\end{equation}
where the choice of the {isotropic} Poisson functional $\zeta$ is discussed below. We aim at building an estimator with smaller mean squared error than the maximum likelihood estimator. By applying Proposition~\ref{lem:specific:case}, we can link the mean squared errors of these two estimators.
\begin{lemma}\label{lem:1}
Let $\zeta\in \dom(\overline{D})$. Consider $\widehat \theta$ defined by~\eqref{eq:defEst}. Then,
\begin{equation} \label{eq:MSEStein0}
\MSE(\widehat \theta) = \MSE(\widehat \theta_{\MLE}) + \frac1{|W|^2} \left( \E(\zeta^2)+ 2\E[ \nabla \zeta] \right).
\end{equation}
\end{lemma}
\begin{proof}
By definition,
\begin{align*}
  \MSE(\widehat \theta) &= \E \left[  \left( \widehat \theta_{\MLE} +
      \frac1{|W|}\zeta-\theta \right)^2  \right] \nonumber \\
  &=\MSE(\widehat \theta_{\MLE}) + \frac1{|W|^2} \left( \E[\zeta^2]+
    2\E[\zeta(N(W)-\theta|W|)] \right).
\end{align*}
Since $\zeta \in \dom(\overline{D})$, we can apply~\eqref{eq:ipp2} with
$F=\zeta$ to deduce~\eqref{eq:MSEStein0}.
\end{proof}
Now, we consider a random variable $\zeta$ written as
$\zeta=\zeta_0\1({N(W)=0}) + \nabla\log(F)$ where $\zeta_0$ is a constant
(possibly zero) and $F$ an almost surely positive Poisson functional
belonging to $\mathcal S$. Both $\zeta_0$ and $F$ are adjusted such that $\zeta
\in \dom(\overline{D})$. Using Lemma~\ref{lem:productgf}, we can follow the initial calculations
of \cite{stein:81}, also used in~\cite{privault:reveillac:09}, to deduce that
$\nabla \log F={\nabla F}/{F}$ and $\nabla\sqrt{F}={\nabla F}/{\com{(2\sqrt{F})}}$.
Using Lemma~\ref{lem:productgf}, we establish the key relations
\begin{equation*}
  \nabla\nabla \log F=
  \frac{F\nabla\nabla
    F-(\nabla F)^2}{F^2} \quad\mbox{ and }\quad
  \nabla\nabla\sqrt{F}=\frac{2\sqrt{F}\nabla\nabla F-(\nabla F)^2/\sqrt{F}}{4F}
\end{equation*}
leading to
\begin{equation}\label{e:key}
  2\nabla\nabla \log F+\left(\frac{\nabla
      F}{F}\right)^2=4\frac{\nabla\nabla\sqrt{F}}{\sqrt{F}}.
\end{equation}
By combining~\eqref{eq:MSEStein0} and~\eqref{e:key}, we obtain the
following result.
\begin{proposition}\label{prop:MSEStein} Let $\zeta_0\in \R$ and $F$ be an almost surely positive Poisson functional such that
  $F$, $\nabla F$  and  $\zeta=\zeta_0 \1(N(W)=0)+ \nabla \log(F) \in \dom(\overline{D})$. Then, the estimator $\widehat\theta$ defined from $\zeta$ by~\eqref{eq:defEst} satisfies
  \begin{equation} \label{eq:MSEStein}
    \MSE(\widehat{\theta})=  \MSE(\widehat\theta_{\MLE}) + \frac1{|W|^2} \left(
      \zeta_0^2 e^{-\theta|W|} + 4\E \left( \frac{\nabla\nabla
          \sqrt{F}}{\sqrt{F}} \right)\right).
  \end{equation}
\end{proposition}
Proposition~\ref{prop:MSEStein} gives a similar result
to~\citet[Proposition~4.1]{privault:reveillac:09} for one--dimensional Poisson
point processes. As a consequence of Proposition~\ref{prop:MSEStein}, the Stein estimator given by~\eqref{eq:defEst} will be more efficient than the
maximum likelihood estimator if we manage to find $F$ and $\zeta_0$ satisfying the conditions of Proposition~\ref{prop:MSEStein}
and such that $\zeta_0^2 e^{-\theta|W|} + 4\E (
{\nabla\nabla \sqrt{F}}/{\sqrt{F}}) <0$. This is investigated in the next section.

\subsection{A class of Stein estimators on the $d$--dimensional Euclidean ball}

In this section, we focus on the case where $W$ is the $d$--dimensional
Euclidean closed ball with center $0$ and radius $w$, denoted
$\overline{B}(0,w)$ in the following. Without loss of generality, we can
restrict to the case $w=1$.  We combine Proposition~\ref{prop:MSEStein} and the
isotropic Malliavin derivative defined in Section~\ref{s:malliavin:derivative}
to build a large class of isotropic Stein estimators. We need some additional notation. Let
$n\geq 1$, $1\leq k\leq n$ and let $x_1,\dots,x_n \in W$, we define $x_{(k),n}$ by induction as follows
\begin{align*}
x_{(1),n} &= \mathrm{argmin}_{u\in \{x_1,\dots,x_n,\}} \|u\|^2 \\
x_{(k),n} &= \mathrm{argmin}_{u\in \{x_1,\dots,x_n\} \setminus
  \{x_{(1),n},\dots ,x_{(k-1),n} \}} \|u\|^2.
\end{align*}
The point $x_{(k),n}$ is the $k-$th closest point of $\{x_1,\dots,x_n\}$
to zero. Similarly, we denote by $X_{(k)}$ the $k-$th closest point of
$\bX$ to zero. Note that, the point $X_{(k)}$ may lie outside $W$ depending on
the value of $N(W)$ for the given realization. We are also given some function $\varphi\in\mathcal{C}^2([0,1],\R^+)$ satisfying the two following additional properties
\begin{equation*}
  \tag{$\mathcal{P}$} \qquad \inf_{t\in [0,1]}\varphi(t)>0 \mbox{ and } \varphi'(1)=0.
\end{equation*}

\noindent Then, the Poisson functional involved in the definition
of our Stein estimator writes
\begin{equation}\label{e:def:Fk}
\sqrt F_k  = \sum_{n \geq k} \1(N(W)=n) g_{k,n}(X_1,\dots,X_n)+ \varphi(1)\,\1(N(W)<k)
\end{equation}
where for $x_1,\dots,x_n \in W$ and $n\geq k \geq 1$
\begin{equation}\label{e:def:form:functions:Fk}
g_{k,n}(x_1,\dots,x_n) = \varphi(\| x_{(k),n}\|^2)
\end{equation}
for a function $\varphi$ satisfying ($\mathcal{P}$). In
other words, we focus on functionals $F \in \mathcal S^\prime$ such that
\begin{align}\label{e:def:Fk:2}
  \sqrt F_k  & = \sum_{n \geq k} \1(N(W)=n) \varphi(\| X_{(k),n}\|^2)+
  \varphi(1)\1(N(W)<k)
\end{align}
from which we build our main result.
\begin{proposition} \label{prop:zeta}
Let $k\geq 1$. Let $\varphi \in \mathcal{C}^2([0,1],\R^+)$
satisfying~($\mathcal{P}$). Define $F_k$
from $\varphi$ as in~\eqref{e:def:Fk:2}. Then $\zeta_k = \nabla \log F_k$ is
an element of $\dom(\overline{D})$. Moreover, the Stein estimator, its
mean squared error and its gain with respect to the
maximum likelihood estimator are given by
\begin{align}
\widehat \theta_k &= \widehat \theta_{\MLE}-\frac{4}{d|W|} \,
\frac{Y_{(k)} \,\varphi'(Y_{(k)})}{\varphi(Y_{(k)} )}\label{e:zetak}\\
\MSE ( \widehat \theta_k) &= {\MSE(\widehat \theta_{\MLE})} - \frac{16}{d^2|W|^2}
\E\left[\mathcal G(Y_{(k)} )\right]\label{e:msek}\\
\Gain(\widehat \theta_k) &= \frac{\MSE(\widehat \theta_{\MLE})-
\MSE(\widehat \theta_k)}{\MSE(\widehat \theta_{\MLE})} =\frac{16}{\theta d^2|W|}\E\left[\mathcal G(Y_{(k)} )\right]
\label{e:gaink}
\end{align}
where
\begin{equation}
  \label{eq:Yk}
  Y_{(k)}=1+(\| X_{(k)}\|^2-1)\1(\| X_{(k)}\|\leq 1)=\left\{
  \begin{array}{ll}
      \| X_{(k)}\|^2&\mbox{ if }\| X_{(k)}\|\leq 1\\
      1&\mbox{ otherwise}
    \end{array}\right.
\end{equation}
and $\mathcal G(t) = - t (\varphi^\prime(t)+t\varphi^{\prime\prime}(t)) / \varphi(t)$.
\end{proposition}
Proposition~\ref{prop:zeta} reveals the interest of the Poisson functional $F_k$ given by~\eqref{e:def:Fk:2}. We obtain isotropic estimators of $\theta$ depending only on $\|X_{(k)}\|^2$. It is worth mentioning that the distribution of $\|X_{(k)}\|^2$ is well--known for a homogeneous
Poisson point process (see Lemma~\ref{lem:gamma}). This allows us to derive efficient and fast estimates of $\E [\mathcal G(Y_{(k)})]$ which can then be optimized w.r.t. the different parameters. This is studied in more details in Section~\ref{sec:sim}.

The proof of Proposition~\ref{prop:zeta} requires first to compute the gradient of the
functions $g_{k,n}$ given by \eqref{e:def:form:functions:Fk} and second to ensure that $\sqrt{F_k}$ belongs to $\dom(\overline{D})$. To this end, we use the following
lemma.
\begin{lemma}\label{lem:deriv:closed}
\com{Let $H^1([0,1],\R)$ be the Sobolev space defined by~\eqref{eq:defSobClosed} and let \linebreak$\Psi \in H^1([0,1],\R)$.} 
Then,
\begin{equation}\label{e:def:F}
G_k=\sum_{n\geq k} \1(N(W)=n)\Psi(\|X_{(k),n}\|^2)+\Psi(1)\1(N(W)<k) \, \in \,\dom(\overline{D})
\end{equation}
and
\begin{equation}\label{e:def:nablaF}
\nabla G_k =-\frac{2}{d}\sum_{n\geq k}\1(N(W)=n)\,\|X_{(k),n}\|^2\Psi'(\|X_{(k),n}\|^2).
\end{equation}
\end{lemma}
The proof of Lemma~\ref{lem:deriv:closed} being quite technical, we postpone it to
Appendix~\ref{s:proof:lem:deriv}. Here, we only present the key ideas sustaining it.\\

{\it Sketch of the proof of Lemma~\ref{lem:deriv:closed}.} Let $n\geq k$ and $x_{n+1}\in \partial W$, i.e. $\|x_{n+1}\|=1$, we have $x_{(k),n+1}=x_{(k),n}$ and $\Psi(\|x_{(k),n+1}\|^2)=\Psi(\|x_{(k),n}\|^2)$,
which is exactly the compatibility condition which has to be satisfied for $n\geq k$ by the form functions of Poisson functionals belonging to $\mathcal{S}'$.
When $n=k-1$ and $x_{n+1}=x_{k}\in \partial W$, i.e. $\|x_k\|=1$, we still have $x_{(k),k}=x_{k}$ and $\Psi(\|x_{(k),k}\|^2)=\Psi(1)$. Since for $n<k$, the forms functions are all equal to the constant function $\Psi(1)$, the compatibility conditions also hold for $n=k-1$ and $n<k-1$.

At any point $(x_1,\dots,x_n)$ such that $(x_1,\dots,x_n)\mapsto
\|x_{(k),n}\|^2$ is differentiable, the definition of the Malliavin derivative
and the usual chain rule easily lead to~\eqref{e:def:nablaF}.  Note that that
even if $\Psi\in \mathcal{C}^1(W,\R)$, the functional $G_k$ may not belong to
$\dom(D)$ since its form functions are not differentiable {everywhere}. Indeed for any $n\geq 1$, $(x_1,\dots,x_n)\mapsto
\|x_{(k),n}\|^2$ is not differentiable at any point $(x_1,\dots,x_n)$ such that
for some $i\neq j$, $\|x_i\|=\|x_j\|=\|x_{(k),n}\|$. In Lemma~\ref{lem:deriv:closed}, we
prove a weaker assertion, namely that $G_k\in\dom(\overline{D})$, which
means that $G_k$ can be obtained as the limit of Poisson functionals of
$\dom(D)$. Then, the proof of Lemma~\ref{lem:deriv:closed} relies on
the density results stated in Appendix~\ref{s:density}.

\begin{proof}[Proof of Proposition~\ref{prop:zeta}]
By definition,
\[
\log F_k  = 2\sum_{n \geq k} \1(N(W)=n)\log\varphi (\| X_{(k),n}\|^2)+\log\varphi(1)\1(N(W)<k).
\]
Since $\varphi$ is a continuously differentiable function, we can easily check that $\Psi\in H^1([0,1])$. So Lemma~\ref{lem:deriv:closed} can be applied to $\Psi=\log\varphi$ and
$G_k=\log(F_k)$. Hence, $\log F_k\in \dom(\overline{D})$ and
\[
\zeta_k=\nabla\log F_k=-\frac{4}{d}\sum_{n\geq k}\1(N(W)=n)\|X_{(k),n}\|^2\frac{\varphi'(\|X_{(k),n}\|^2)}{\varphi(\|X_{(k),n}\|^2)}.
\]
Then, we derive the explicit expression of $\widehat{\theta}_k$ given
by~\eqref{e:zetak}.  We also deduce that $\zeta_k=\nabla\log F_k\in
\dom(\overline{D})$ by applying once more Lemma~\ref{lem:deriv:closed}
with $\Psi(t)=t{\varphi'(t)}/{\varphi(t)}$, which also satisfies the required
properties thanks to the smoothness of $\varphi'$.  In view of
Proposition~\ref{prop:MSEStein},  we estimate
$\nabla\nabla\sqrt{F}_k/\sqrt{F}_k$ to derive~\eqref{e:msek}.
From~\eqref{e:def:Fk} and \eqref{e:def:form:functions:Fk}, $\sqrt{F}_k$ also
satisfies the assumptions of Lemma~\ref{lem:deriv:closed} with $\Psi=\varphi$.
Hence $\sqrt{F}_k\in \dom(\overline{D})$ and
\[
\nabla\sqrt{F}_k=-\frac{2}{d}\sum_{n\geq k}\1(N(W)=n)\|X_{(k),n}\|^2\varphi'(\|X_{(k),n}\|^2).
\]
The conclusion ensues by applying Lemma~\ref{lem:deriv:closed} to
$G_k=\nabla\sqrt{F}_k$ with $\Psi(t)=-2t\varphi'(t)/d$ and we obtain the following formulae.
\begin{align*}
\nabla\nabla\sqrt{F}_k&=\frac{4}{d^2}\sum_{n\geq
  k}\1(N(W)=n)\left[\|X_{(k),n}\|^2\varphi'(\|X_{(k),n}\|^2)+\|X_{(k),n}\|^4\varphi''(\|X_{(k),n}\|^2)\right]\\
\frac{\nabla\nabla\sqrt{F}_k}{\sqrt{F}_k}&=\frac{4}{d^2}\sum_{n\geq k}\1(N(W)=n)\frac{\left[\|X_{(k),n}\|^2\varphi'(\|X_{(k),n}\|^2)+\|X_{(k),n}\|^4\varphi''(\|X_{(k),n}\|^2)\right]}{\varphi(\|X_{(k),n}\|^2)}.
\end{align*}
Then~\eqref{e:msek} follows from the last equality, while~\eqref{e:gaink} is
directly deduced from~\eqref{e:msek}.
\end{proof}

\section{Numerical experiments} \label{sec:sim}

We underline that it is impossible to find a function $\varphi$ satisfying $(\mathcal P)$ and such that $\mathcal G(t)$ defined by~\eqref{e:gaink} is positive for any $t\in [0,1]$. In this section, we analyze two examples for which we can obtain positive gain even though $\mathcal G$ is not positive everywhere. Then, we conduct a numerical and simulation study where, in particular, we show that in many practical situations we can get a gain larger than 30\%.

Before this, we should notice that the mean squared error
and the gain of our new estimators $\widehat{\theta}_k$ only depend on the
distribution of $\norm{X_{(k)}}^2$. The following result shows that expectations
involving this random variable can be computed directly without sampling a Poisson
point process, which speeds up the simulations a lot.

\begin{lemma} \label{lem:gamma}
  Let $k\geq 1$, the density  of $\norm{X_{(k)}}^2$ is given by
  \[
    f_{\norm{X_{(k)}}^2}(t) = \frac{d}{2} v_d \theta t^{-1} e^{v_d \theta t^{d/2}} \,
    \frac{\left( v_d\theta t^{d/2}\right)^{k-1}}{(k-1)!} \1(t \ge 0)
  \]
  where $v_d = \pi^{d/2}\Gamma(d/2+1) w^d$ is the volume of $\overline{B}(0, w)$.
  Moreover, for any positive measurable function $h:\R^+ \to \R^+$, we have
  $\E \left( h(\norm{X_{(k)}}^2) \right) = \E \left( h(Z^{2/d}) \right)$
  where $Z$ is a real valued random variable following a Gamma distribution with shape $k$
  and rate~$v_d \theta$.
\end{lemma}

\subsection{First example}

Let $\kappa>0$ and $\gamma\in (0,1/2)$, we define for $t\in [0,1]$
\begin{equation}
  \label{e:ch1} \varphi(t)=(1-t) (\chi_{[0,1-\gamma]} * \psi)(t) \;\,
   +\kappa
\end{equation}
where, for any measurable set $A$, $\chi_A = \1(t \in A)$ denotes the
characteristic function of the set $A$ and the star $*$ stands for the
convolution product. The Schwarz function (see e.g. \cite{hormander03}) $\psi:[-1,1]\to \R^+$, defined by
\[
  \psi(t) = c \exp\left( -\frac1{1-|t|}\right) \mbox{ with } c \mbox{ such that} \int_0^1\psi(t)\dd t=1
\]
satisfies $\psi\geq 0$ and $\psi^{(m)}(\pm 1)=0$ for any $m\geq 0$, which implies that $\varphi$ satisfies
($\mathcal P$). The main interest of this function is
that for any $t\in [0,1-2\gamma]$, $\varphi(t)=1-t+\kappa$,
$\varphi^\prime(t)=-1$ and $\varphi^{\prime\prime}(t) =0$ which leads to
$\mathcal G(t)=t/(1-t+\kappa)\geq 0$ for any $t\in [0,1-2\gamma]$. Figure~\ref{fig:ch1}
illustrates this function $\varphi$ with $\kappa=0.5$ and $\gamma=0.05$. We can
observe that $\mathcal G(t)\geq 0$ for $t\leq 0.9$ but $\cG(t)$ can be
highly negative for $t>0.9$. Note also that the smaller $\gamma$, the more
negative $\mathcal G$. This highlights the danger of Example~1. From a
practical point of view, the best choice would be to tune the integer $k$ such
that $\|X_{(k)}\|^2$ often lies in a region in which $\cG$ is high, however this
region is quite small and the function $\cG$ decreases quickly outside of it and
takes highly negative values, which may yield to negative gains in the end.  On
the contrary, if reasonable values are chosen for $\gamma$ and $k$ is small then
there is hardly no chance that $\|X_{(k)}\|^2>1-2\gamma$ but the corresponding
gain value remains very small.  This first example shares some similarities with
the estimator proposed by \cite{privault:reveillac:09} in the case $d=1$, see
Section~\ref{sec:PR}.

\begin{figure}[htbp]
\begin{center}
\includegraphics[scale=.5]{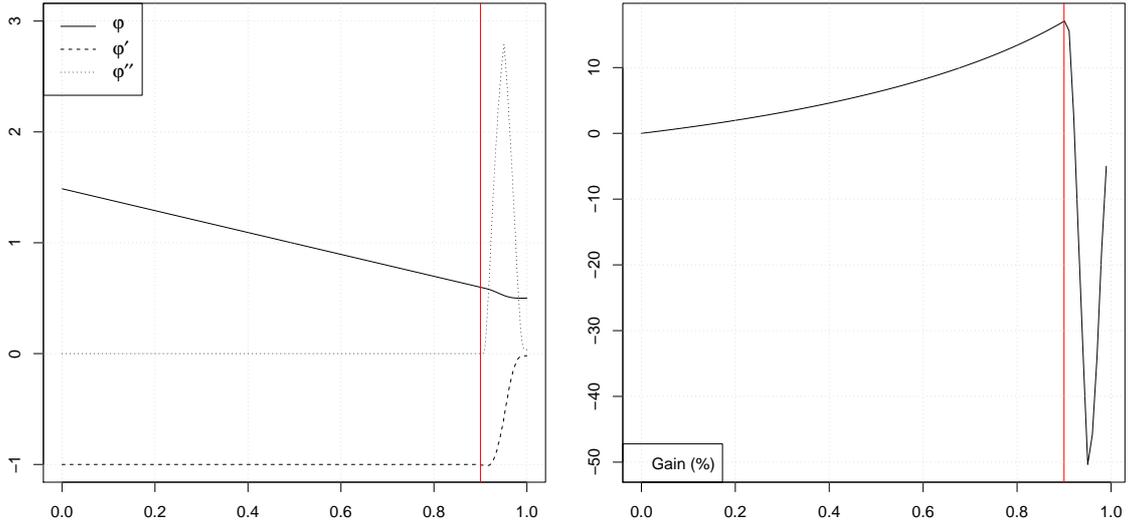}
\caption{\label{fig:ch1}Plots of $\varphi,\varphi^\prime, \varphi^{\prime\prime}$ and $\mathcal G$ for the function $\varphi$ defined by~\eqref{e:ch1} with $\kappa=0.5$ and $\gamma=0.05$. The vertical line indicates the value $1-2\gamma=0.9$ before which the function $\varphi$ is linear.}
\end{center}
\end{figure}

\subsection{Second example} Let $\kappa\geq 2$ and $\gamma\in \R$, we define for $t\in [0,1]$
\begin{equation}
  \label{e:ch2} \varphi(t) = \exp\left( \gamma(1-t)^\kappa\right).
\end{equation}
We can easily check that the property~($\mathcal P$) holds
for any $\kappa\geq 2$.  The main advantage of this function is that the gain
function has a ``power'' shape. For instance, when $\kappa=2$, $\mathcal G(t) =
2\gamma t\left( 1-2t - 2 \gamma t (1-t)^2 \right)$. For any value of $\kappa$, we
can show that there exists a unique $t_0\in (0,1)$ such that $\mathcal G(t_0)=0$
and such that $\sign(\mathcal G)=\sign(\gamma)$ for $t\in [0,t_0)$ and
$\sign(\mathcal G)= \sign(-\gamma)$ for $t\in (t_0,1]$. Figure~\ref{fig:ch2}
illustrates this function. The top-right figure reveals the interest of
Example~2. It shows that when $k$ is chosen large enough, then
$\|X_{(k)}\|^2>t_0$ very often and therefore $\mathcal G(t)$ is quite large.

The latter comment is the main reason why among these two examples (and many
others we have tried) the exponential function~\eqref{e:ch2} produces the best
 and the most stable empirical results. Thus, we only focus on this case in the
following. With this choice for the function $\varphi$, the Stein estimator
writes
\begin{align}
   \widehat \theta_k &    = \widehat \theta_{\MLE} + \frac{4d}{|W|} \gamma \kappa
  Y_{(k)}  (1-Y_{(k)})^{\kappa-1}  \label{e:steinCh2}
\end{align}
where $Y_{(k)}$ is given by \eqref{eq:Yk}.

\subsection{Optimization of the gain and a first simulation study}

\begin{figure}[htbp]
\begin{center}
\includegraphics[scale=.5]{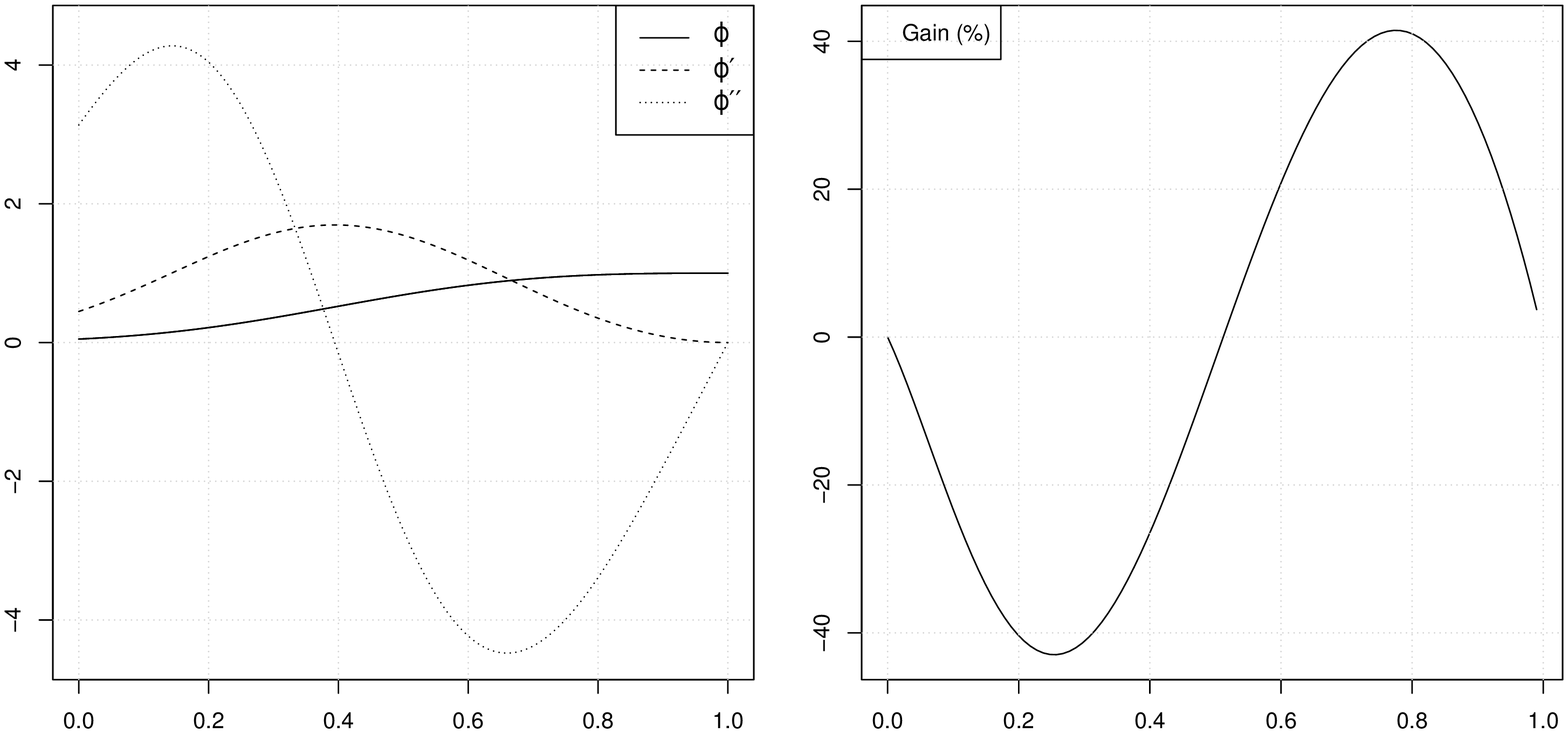}
\includegraphics[scale=.5]{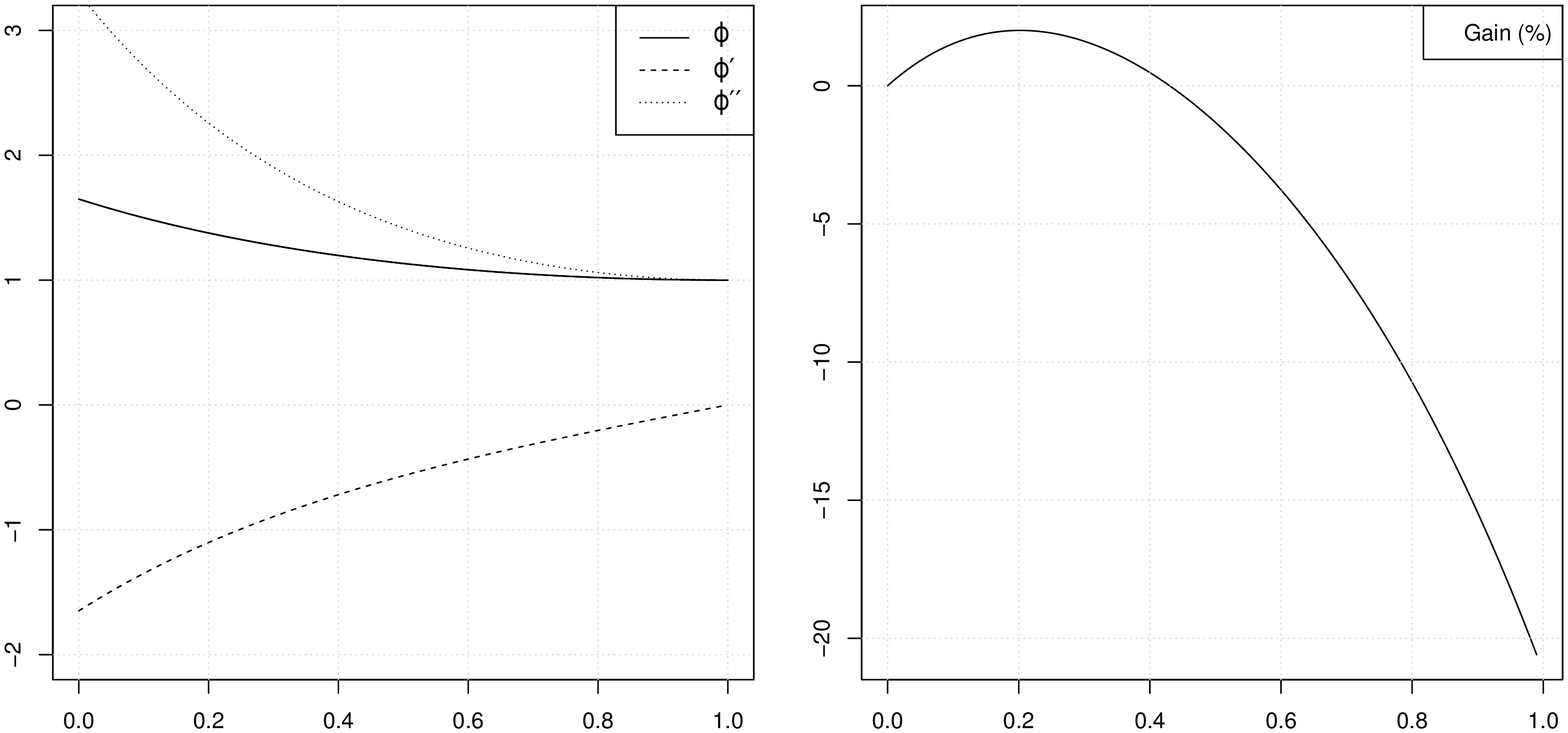}
\caption{\label{fig:ch2}Plots of $\varphi,\varphi^\prime, \varphi^{\prime\prime}$ and $\mathcal G$ for the function $\varphi$ defined by~\eqref{e:ch2}. The parameters $\kappa$ and $\gamma$ are equal to $3$ and $-3$ for the top figures and to 2 and 0.5 respectively for the bottom figures.}
\end{center}
\end{figure}

Before optimizing the parameters $k,\kappa$ and $\gamma$, we first want to check
the integration by parts formula. For $k=10,20,50$ and $80$ (and for specific
parameters $\kappa$ and $\gamma$ we do not comment upon for now), we represent in
Figure~\ref{fig:ipp} the empirical and theoretical gains computed by the Monte-Carlo
approximation of~\eqref{e:gaink} based on $50000$ replications of Poisson point
processes in the 2--dimensional Euclidean ball. We clearly observe that the
empirical gain perfectly fits the one computed from~\eqref{e:gaink}. The second
observation is that $k$, $\kappa$ and $\gamma$ need appropriate tuning,
otherwise the gain can become highly negative. For instance, when $k=20$ and
$\theta=20$ (and $\kappa$, $\gamma$ chosen as specified in the caption of
Figure~\ref{fig:ipp}) the gain reaches the value -200\%.

\begin{figure}[htbp]
\begin{center}
\includegraphics[scale=.6]{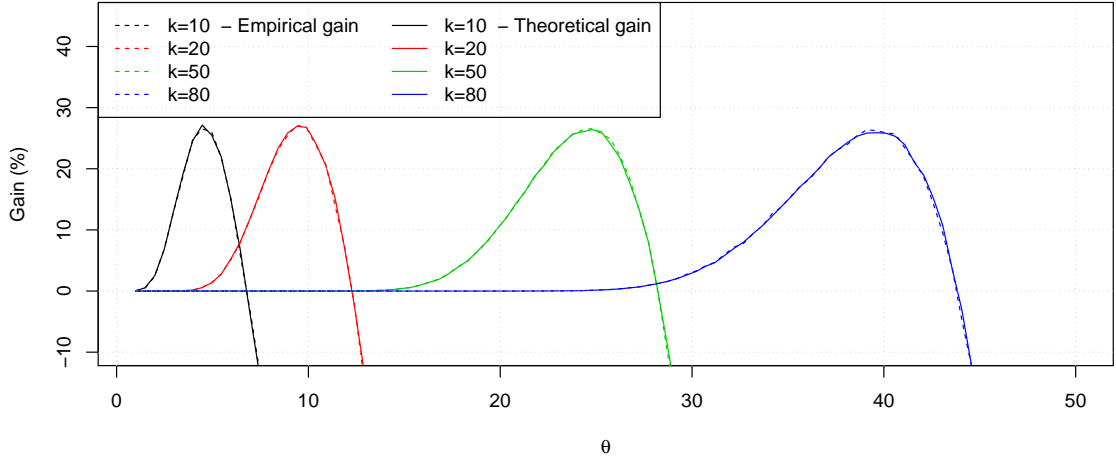}
\caption{\label{fig:ipp}Empirical and Monte-Carlo approximations of the
  theoretical gain in terms of $\theta$ for the exponential family
  \eqref{e:ch2}. For $k=10,20,50,80$, the parameters $\kappa$ and $\gamma$ are
  chosen to optimize the theoretical gain for $\theta=5,10,25,40$ respectively.
  The results are based on $m=50000$ replications of Poisson point processes in
  the 2--dimensional euclidean ball $\overline{B}(0,1)$. The Monte-Carlo approximation
  of~\eqref{e:gain-ch2} is also based on $m=500000$ samples.}
\end{center}
\end{figure}

For the exponential function $\varphi$ given by~\eqref{e:ch2}, the gain function writes
\begin{align*}
  \cG(t; \kappa; \gamma) =  \gamma \kappa t (1 - t)^{\kappa-1} - \gamma^2
  \kappa^2 t^2 (1 - t)^{2(\kappa-1)} - t^2 \gamma \kappa(\kappa-1) (1-t)^{\kappa-2}.
\end{align*}
Now, we explain how we can compute $\arg\max_{(k, \gamma, \kappa)}\E(\cG(Y_{k}; \kappa; \gamma))$. First, note that solving the equation $\frac{\partial}{\partial\gamma} \cG(t; \kappa; \gamma) = 0$ leads to an explicit formula for the optimal choice of the parameter $\gamma$
\begin{align*}
  \gamma^\star = \frac{\E\left( Y_{(k)} (1 -
      Y_{(k)})^{\kappa-1} - Y_{(k)}^2 (\kappa - 1) (1 -
      Y_{(k)})^{\kappa-2} \right)}{\E\left( 2 \kappa {Y_{(k)}}^2
      (1-Y_{(k)})^{2(\kappa-1)}\right)}.
\end{align*}
Plugging this value back in the formula of the gain leads to
\begin{align}
  \label{e:gain-ch2}
  \E(\cG(Y_{(k)}; \kappa; \gamma^\star)) =
  \frac{\E\left(  Y_{(k)} (1-Y_{(k)})^{\kappa-2} (1 - \kappa
      Y_{(k)}) \right)}{\E\left( {Y_{(k)}}^2 (1 -
      Y_{(k)})^{2(\kappa-1)} \right)}.
\end{align}
Second, we   compute numerically $\arg\max_\kappa = \E(\cG(Y_{(k)};
\kappa; \gamma^\star))$. To do so, we rely on deterministic optimization
techniques after replacing the expectation by a sample average; we refer the
reader to~\cite{rubinstein:shapiro:93} for a review on sample average
approximation. Note that the random variable $Y_{(k)}$ can be sampled very
efficiently by using Lemma~\ref{lem:gamma} and~\eqref{eq:Yk}.

In Figure~\ref{fig:kfixed}, we chose different values of $k$ and optimized,
w.r.t $\kappa$ and $\gamma$ for every value of $\theta$, the theoretical
gain~\eqref{e:gaink} computed by Monte-Carlo approximations. The plots are
presented in terms of $\theta$ (when $d=2$). For a fixed value of $k$, we
observe that when $\gamma$ and $\kappa$ are correctly adjusted the gain is
always positive for any $\theta$. Still, the choice of $k$ is very sensitive to
the value of $\theta$ and also needs to be optimized to reach the highest gain.
This has been done in Table~\ref{tab:sim}, in which we present a first
simulation study. We investigate the gains for different values of $d$ and $\theta$. For any
$\theta$ and $d$, we chose
\[
(k^\star,\gamma^\star,\kappa^\star)= \mathrm{argmax}_{(k,\gamma,\kappa)} \Gain(\widehat \theta_k).
\]
Our experience is that interesting choices for $k$ are values close to the
number of points, say $n$. Therefore, the optimization has been done for $k \in
\{\lfloor .75n \rfloor,\dots, \lfloor 1.2n\rfloor \}$. Such an optimization is
extremely fast. Using 50000 samples to approximate~\eqref{e:gaink} with the
help of Lemma~\ref{lem:gamma}, it takes less than two seconds to find the
optimal parameters when $d=3$ and $\theta=40$. The empirical results presented
in Table~\ref{tab:sim} are based on 50000 replications. We report the empirical
means, standard deviations, MSE for both the MLE and the "optimized" Stein
estimator and finally the empirical gain. The average number of points in
$\overline B(0,1)$ equals to $\theta|W|$ with $|W|=2,3.14,4.19$ approximately.  The
first three columns allow us to recover that the MLE is unbiased with variance
(and thus MSE) equal to $\theta/|W|$. Then, we observe that our Stein estimator
is always negatively biased. This can be seen from~\eqref{e:steinCh2} since the
optimal value $\gamma^\star$ is always negative. We point out that the standard
deviation is considerably reduced which enables us to obtain empirical
gain between 43\% and 48\% for the cases considered in the simulation.

\begin{figure}[htbp]
\begin{center}
\includegraphics[scale=.6]{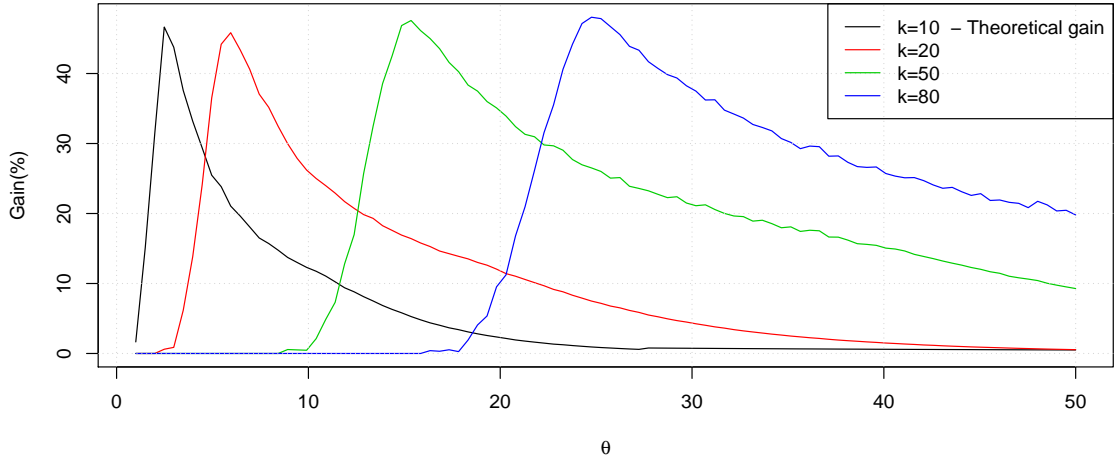}
\caption{\label{fig:kfixed}Monte-Carlo approximation of the theoretical gain
  for the exponential family~\eqref{e:ch2}. For each
  value of $k$ and $\theta$, the parameters $\kappa$ and $\gamma$ are chosen to
  optimize the theoretical gain. The  Monte-Carlo approximations
  of~(\ref{e:gain-ch2}) are based on $m=50000$ replications of Poisson point processes
  in the 2--dimensional euclidean ball $\overline{B}(0,1)$.}
\end{center}
\end{figure}

\begin{table}[ht]
\begin{center}
\begin{tabular}{r|rrr|rrrr|r}
  \hline
& \multicolumn{3}{|c}{\sc mle} &\multicolumn{4}{|c|}{\sc stein} &   Gain (\%) \\
& \multicolumn{1}{|r}{mean} & sd & \multicolumn{1}{r|}{mse} & $k^\star$ & mean & sd & \multicolumn{1}{r|}{mse} & \\
  \hline
$\theta=5$, $d=1$ & 5 & 1.6 & 2.52 & 11 & 4.4 & 1.0 & 1.44 & 43.0 \\
$d=2$ &  5 & 1.3 & 1.58 & 18 & 4.6 & 0.8 & 0.86 & 45.6 \\
  $d=3$ &  5 & 1.1 & 1.19 & 22 & 4.6 & 0.7 & 0.64 & 46.1 \\
  $\theta=10$, $d=1$ & 10 & 2.2 & 5.03 & 22 & 9.2 & 1.4 & 2.73 & 45.8 \\
  $d=2$ &10 & 1.8 & 3.18 & 34 & 9.4 & 1.2 & 1.72 & 46.0 \\
  $d=3$ & 10 & 1.5 & 2.37 & 44 & 9.5 & 1.0 & 1.27 & 46.3 \\
  $\theta=20$, $d=1$ &20 & 3.1 & 9.91 & 42 & 18.8 & 2.0 & 5.31 & 46.4 \\
  $d=2$ &20 & 2.5 & 6.38 & 66 & 19.1 & 1.6 & 3.41 & 46.5 \\
  $d=3$ &20 & 2.2 & 4.72 & 84 & 19.1 & 1.3 & 2.47 & 47.5 \\
 $\theta=40$, $d=1$  &40 & 4.5 & 20.09 & 84 & 38.5 & 2.9 & 10.61 & 47.2 \\
  $d=2$ & 40 & 3.6 & 12.79 & 125 & 38.6 & 2.2 & 6.78 & 46.9 \\
  $d=3$ & 40 & 3.1 & 9.58 & 169 & 38.8 & 1.9 & 4.95 & 48.3 \\
   \hline
\end{tabular}
\caption{\label{tab:sim} Empirical means, standard deviations (sd), mean squared errors (mse) of {\sc mle} estimators and {\sc stein} estimators of the intensity parameter of Poisson point processes observed in the $d$--dimensional euclidean ball $\overline{B}(0,1)$. The results are based on $m=50000$ replications. The Stein estimator is based on the exponential function. For each $\theta$ and $d$, the parameters $\kappa$, $\gamma$ and the $k-$th nearest--neighbor are optimized to maximize the theoretical gain. The column $k^\star$ reports the optimal nearest-neighbour. Finally, the last column reports the gain of mse in percentage of the Stein estimator with respect to the {\sc mle}.}
\end{center}
\end{table}

\subsection{Comparison with Privault-Réveillac's estimator in the case $d=1$}
\label{sec:PR}

In this section, we focus on the case $d=1$ and we aim at comparing the
performance of our estimator with the one proposed by
\cite{privault:reveillac:09} and denoted $\widehat \theta_{\PR}$ for the sake of
conciseness. As underlined previously, $\widehat \theta_{\PR}$ shares some
similarities with our first example. The main difference comes from the fact that,
even in the case $d=1$, the integration by parts formula obtained by
\citet[Proposition 3.3]{privault:reveillac:09} differs from ours (see
Theorem~\ref{thm:ipp}). Since our framework was to work with
$d$--dimensional Poisson point processes for any $d\geq 1$, we had to impose  different
compatibility conditions. To ease the
comparison with our estimator based on  $\bX$ defined on $\R$ and observed on
$W=[-1,1]$, we define $\widehat \theta_{\PR}$ based on the observation of $\bX$
on $\widetilde W=[0,2]$. Note that by stationarity,
$(\bX_W+1)\stackrel{d}{=}\bX_{\widetilde W}$ so both  estimators are based on
the same amount of information. Let $X_1$ be the point of $\bX$ closest to 0,
then $\widehat \theta_{\PR}$ is defined for some $\kappa>0$ by
\begin{equation*}
     \widehat\theta_{\PR} = \widehat \theta_{\MLE} + \frac2\kappa \1 (N(\widetilde W)=0) \; + \; \frac{2X_1}{2(1+\kappa)-X_1} \; \1( 0<X_1\leq 2).
  \end{equation*}
The mean squared error and the corresponding gain are given by
\begin{align*}
\MSE(\widehat\theta_{\PR} ) &= \MSE (\widehat \theta_{\MLE}) + \frac1{\kappa^2}\exp(-2\theta) - \E \left(
\frac{X_1}{2(1+\kappa)-X_1} \1 (0<X_1\leq 2) \right)
\end{align*}
\begin{equation}\label{eq:gainPR}
  \Gain (\widehat\theta_{\PR}) = \frac{2}{\theta \kappa^2} \exp(-2\theta) \; - \;
  \frac2\theta \E \left( \frac{X_1}{2(1+\kappa)-X_1} \1 (X_1\leq 2) \right).
\end{equation}
Note that $X_1 \sim \mathcal E(\theta)$. We took the same point of view as in the previous section and optimized the gain w.r.t. $\kappa$. The optimal gain reached by $\widehat \theta_{\PR}$ is presented in Figure~\ref{fig:gainPR}. As a summary of this curve, the optimal gain for $\theta=5,10,20$ and $40$ is equal to $4.4\%, 1.1\%, 0.2\%$ and $0.06\%$ respectively. The results are clear. Our Stein estimator, based on the exponential function and on the idea of picking the $k$-th closest point to $0$ instead of just the first one, considerably outperforms the estimator proposed by \cite{privault:reveillac:09}.

\begin{figure}[htbp]
\centering\includegraphics[scale=.6]{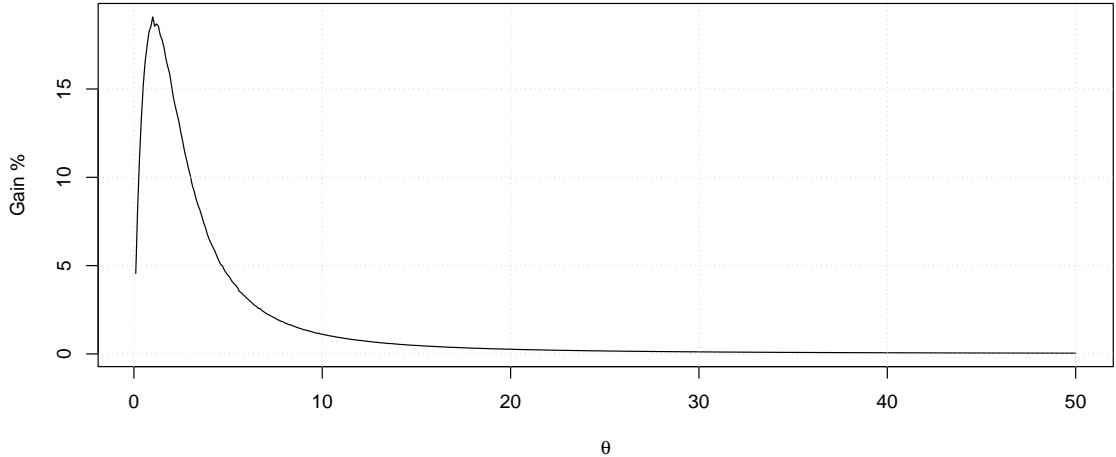}
\caption{\label{fig:gainPR} Monte-Carlo approximation of the theoretical  gain~\eqref{eq:gainPR} optimized in terms of $\kappa$ of the Stein estimator $\widehat \theta_{\PR}$ in terms of $\theta$.}
\end{figure}

\subsection{Data-driven Stein estimator}

Table~\ref{tab:sim} is somehow a theoretical table since the optimal parameters
$k^\star$, $\gamma^\star$ and $\kappa^\star$ are searched for given the value
of $\theta$, which is useless in practice since $\theta$ is unknown.  A
natural idea consists in first estimating the MLE and then look for the optimal
parameters given $\widehat\theta_{\MLE}$. Preliminary experiments have shown us
that this procedure can be quite varying when $d$ or $\theta$ are large. We
think that this is essentially due to the high variance of the MLE. To explain how
we can reduce this, let $f_{k,\gamma,\kappa}(\theta)= 16 \E
\cG(Y_{(k)})/(\theta d^2|W|)$. Instead of maximizing
$f_{k,\gamma,\kappa}(\widehat \theta_{\MLE})$, we suggest to maximize the
average gain for a range of values of $\theta$ and we fix this range as a
factor of the standard deviation of the MLE. More specifically, let
\[
  \Theta(\theta,\rho) = \left[ \theta - \rho  \sqrt{\theta/|W|},
\theta + \rho  \sqrt{\theta/|W|}
  \right].
\]
When $\rho=1$ (resp. $\rho=1.6449$, $\rho=1.96$),
$\Theta(\widehat{\theta}_{\MLE},\rho)$ corresponds to the confidence interval of
$\theta$ based on $\widehat \theta_{\MLE}$ with confidence level 68.3\% (resp.
$90\%$ and $95\%$). Then, we suggest to maximize
\begin{equation}
  \label{e:gainInt} \int_{\Theta(\widehat{\theta}_{\MLE},\rho)} f_{k,\gamma,\kappa}(\theta) \dd \theta = \frac{16}{d^2|W|}
  \E \int_{\Theta(\widehat\theta_{\MLE},\rho)} \frac{\cG( Y_{(k)})}{\theta} \dd \theta.
\end{equation}
In the following, we may write $Y_{(k)}(\theta)$ to emphasize
that the distribution of the random variable $Y_{(k)}$ depends on
the parameter $\theta$. Thus, we can rewrite~\eqref{e:gainInt} as
\begin{equation*}
   \int_{\Theta(\widehat{\theta}_{\MLE},\rho)} f_{k,\gamma,\kappa}(\theta) \dd
   \theta = \frac{16}{d^2|W|} 2 \rho \sqrt{ \widehat{\theta}_{\MLE}/ |W|}
   \E \left(\frac{\cG(Y_{(k)}(U))}{U}\right),
\end{equation*}
where $U$ is a random variable independent of $Y_{(k)}$ with uniform distribution over
$\Theta(\widehat{\theta}_{\MLE},\rho)$. Sampling $Y_{(k)}(U)$ is
performed in two steps: first, sample $U$ and second sample a Gamma distributed
random variable as explained in Lemma~\ref{lem:gamma}, in which the value of
$\theta$ is replaced by the current sample of $U$. Hence, optimizing this new
criteria basically boils down to the same kind of computations as for a fixed
value of $\theta$ without bringing in any extra computations.

Table~\ref{tab:sim2} reports the empirical results regarding this suggestion. As
the optimization procedure needs to be conducted for each replication, we only
considered 5000 replications. We report the results of the empirical gains of
the previous procedure for $\rho=0,1,1.6449$ and $1.96$, the value $0$ meaning
that we simply maximize $f_{k,\gamma,\kappa}(\widehat\theta_{\MLE})$.  Globally,
the empirical gains are slightly smaller than the ones obtained in
Table~\ref{tab:sim}. Yet, the results remain pretty impressive. When $\theta$
is equal to $5$ or $10$, optimizing \eqref{e:gainInt} with $\rho=1$ leads to
similar results as the previous ones. The value $\rho=1$ seems again to be a
good choice when $\theta=20$ while $\rho=1.6449$ is a good compromise when
$\theta=40$.

\begin{table}[ht]
\begin{center}
\begin{tabular}{r|rrrr}
  \hline
& \multicolumn{4}{c}{Gain (\%)}\\
& $\rho=0$ & $\rho=1$ & $\rho=1.6449$ & $\rho=1.96$ \\
  \hline
$\theta=5$, $d=1$ &48.8 & 47.9 & 36.4 & 30.1 \\
$d=2$ &  38.6 & 42.4 & 37.1 & 31.4 \\
  $d=3$ &39.4 & 42.6 & 37.0 & 31.7 \\
  $\theta=10$, $d=1$ &40.3 & 43.8 & 36.7 & 30.1 \\
  $d=2$ &36.2 & 38.8 & 33.7 & 27.9 \\
  $d=3$ &31.6 & 36.6 & 32.0 & 28.3 \\
  $\theta=20$, $d=1$ & 37.3 & 38.6 & 34.5 & 28.0 \\
  $d=2$ & 27.3 & 33.1 & 31.0 & 26.5 \\
  $d=3$ &20.8 & 28.6 & 28.1 & 23.8 \\
 $\theta=40$, $d=1$  & 22.3 & 30.8 & 29.2 & 23.9 \\
  $d=2$ &  16.3 & 24.0 & 28.2 & 24.4 \\
  $d=3$ & 12.7 & 19.0 & 24.5 & 22.0 \\
   \hline
\end{tabular}
\caption{\label{tab:sim2} Empirical gain of Stein estimators of the intensity parameter of Poisson point processes observed in the $d$--dimensional euclidean ball $\overline{B}(0,1)$. The results are based on $m=5000$ replications. The Stein estimator is based on the exponential function. For each replication, the parameters $k,\gamma$ and $\kappa$ maximize~\eqref{e:gainInt} for different values of $\rho$.}
\end{center}
\end{table}

\appendix

\section{Notation}

We introduce some additional notation required in the proofs. Let $m$ be a fixed integer. We denote by $\mathcal{D}(\mathcal{O},\R)$ the space of compactly supported functions which are infinitely differentiable on an open subset $\mathcal{O}$ of $\R^m$. If $\overline{\mathcal{O}}$ is a closed subset of $\R^m$, we define
\[
\mathcal{D}(\overline{\mathcal{O}},\R)=\big\{g:\,\exists \widetilde{g}\in
  \mathcal{D}(\R^m,\R)\mbox{ such that }g=\widetilde{g}_{|\overline{\mathcal{O}}}\big\}.
\]
The Sobolev spaces on an open subset $\mathcal{O}$ of $\R^m$ are defined by
\[
H^1(\mathcal{O},\R)=\left\{f\in L^2(\mathcal{O},\R): \,\forall 1 \le i \le m,\,\frac{\partial f}{\partial x_i}\in L^2(\mathcal{O},\R)\right\}
\]
with the norm defined for any $g \in H^1(\mathcal{O},\R)$ by
\[
  \|g\|_{H^1(\mathcal{O},\R)}=\|g\|_{L^2(\mathcal{O},\R)}+\sqrt{\sum_{i=1}^m\left\|\frac{\partial g}{\partial x_i}\right\|^2_{L^2(\mathcal{O},\R)}}
\]
where the partial derivatives have to be understood in the weak sense. The Sobolev spaces can also be defined on a closed subset $\overline{\mathcal{O}}$ of $\R^m$ as
\begin{equation}\label{eq:defSobClosed}
H^1(\overline{\mathcal{O}},\R)=\{f\in L^2(\overline{\mathcal{O}},\R): \, \exists g\in H^1(\R^m,\R)\mbox{ such that }g_{|\overline{\mathcal{O}}}=f\}.
\end{equation}
\section{Proof of Lemma~\ref{lem:deriv:closed}}\label{s:proof:lem:deriv}
\subsection{An explicit formula for the weak gradient of $\Psi(\|x_{(k),n}\|^2)$}

In this section, we consider form functions defined by
\[
  f_n(x_1,\dots,x_n)= \Psi(\|x_{(k),n}\|^2).
\]
Even if $\Psi\in \mathcal{C}^1(\R,\R)$, the function $f_n$ may not differentiable everywhere since the function $(x_1,\dots,x_n)\mapsto \|x_{(k),n}\|^2$ is non differentiable at points $(x_1,\dots,x_n)$ for which $\|x_i\|=\|x_j\|=\|x_{(k),n}\|$ for some $i \neq j$. Nevertheless, $(x_1,\dots,x_n)\mapsto \|x_{(k),n}\|^2$ is continuous on $(\R^d)^n$ and so is $f_n$. Then, we deduce in this section that $f_n$ admits partial {weak derivatives} (see Lemma~\ref{lem:deriv}). Our result is based on the following classical result concerning the existence
of weak derivatives for continuous functions (see e.g.\citet[Proposition 2.4 of Chapter 3]{zuily:02}).
\begin{lemma}\label{lem:jump}
Let $(a_n)_{n\in\N}$ be an increasing sequence of real numbers
such that $a_n\to \infty$ as $n\to \infty$. Set $a_0=-\infty$. For any $j\geq
0$, let $f_j : [a_j,a_{j+1}] \to \R$ be such that $f_j\in\mathcal{C}^1((a_j,a_{j+1}),\R),\,f'_j\in L^1((a_j,a_{j+1}),\R)$ and define $f$  by $f=\sum_{j\in\N} f_j\1_{[a_j,a_{j+1})}$. If the function $f$ is continuous on $\R$, then it admits a weak derivative, denoted $f'$, defined as the locally integrable function $f'= \sum_{j}f'_j\1_{(a_j,a_{j+1})}$.
\end{lemma}
\begin{remark}
The continuity assumption on $f$ is crucial since if $f$ were discontinuous at
some points $a_i$, $f$ would not admit a weak derivative. Indeed, in this case
the usual {jump discontinuity formula} (see again~\citet[Proposition 2.4 of Chapter 3]{zuily:02}) would
imply that the derivative of $f$, {in the sense of distributions}, would be the
sum of a locally integrable function and some Dirac masses.
\end{remark}
\noindent We deduce the following result from Lemma~\ref{lem:jump}.
\begin{lemma} \label{lem:deriv}
Let $n\geq k\geq 1$ and let $\Psi:\R \to \R$ be a continuously differentiable function defined on $\R$. For any $j$, $f_n$ admits a weak gradient $\nabla_{x_j} f_n(x_1,\dots,x_n)$ w.r.t. $x_j$ and the following equality holds
\begin{equation}\label{e:equality:distrib}
\sum_{j=1}^n \nabla_{x_j} f_n(x_1,\dots,x_n) \cdot x_j=2\|x_{(k),n}\|^2  \Psi^\prime(\|x_{(k),n}\|^2), \; a.e.
\end{equation}
In addition, if $\Psi$ is compactly supported in $[-\sqrt{a},\sqrt{a}]$ for some $a>0$, the function $f_n$ belongs to $H^1((\overline{B}(0,a))^n,\R)$ and satisfies $\sup_n(\|f_n\|_{L^\infty(\R^n)}+\sum_i\|\nabla_{x_i}f_n\|_{L^\infty(\R^n)})<\infty$.
\end{lemma}
\begin{proof}
Define for any $j\in \{1,\dots,n\}$, the following set
\[
A_j=\{(x_1,\dots,x_n)\in (\R^d)^n: \,x_{(k),n}=x_j\mbox{ and } \|x_j\|\neq\|x_i\| \mbox{ for }j\neq i\}.
\]
Observe that on $\bigcup_j A_j$
\[
f_n(x_1,\dots,x_n)=\sum_j \Psi(\|x_j\|^2)\1_{A_j}(x_1,\dots,x_n).
\]
Hence, the everywhere differentiability of $\Psi$ implies that the function
$f_{n}$ is differentiable on $\bigcup_j A_j$. In addition, the usual
differentiability rules lead to
\begin{align*}
\nabla_{x_i}f_n(x_1,\dots,x_n)&=\sum_{j=1}^n \nabla_{x_i}[\Psi(\|x_j\|^2)\1_{A_j}(x_1,\dots,x_n)]\\
&=\nabla_{x_i}[\Psi(\|x_i\|^2)\1_{A_i}(x_1,\dots,x_n)]\\
&=2x_i \Psi'(\|x_i\|^2)\1_{A_i}(x_1,\dots,x_n)
\end{align*}
for any $i\in\{1,\dots,n\}$. By using Lemma~\ref{lem:jump}, we prove that $f_n$ admits a weak gradient w.r.t. $x_i$. In the following, we denote the coordinates of elements $x_i \in \R^d$ by $x_i = (x_i^{(1)}, \dots, x_i^{(d)})^\top$. For any $\ell=1,\dots,d$, define
\[
  f_{i,n}^{(\ell)}(x_i^{(\ell)})=f(x_1,\dots,x_{i-1},(x_i^{(1)},\dots,x_i^{(\ell)},\dots,x_i^{(d)}),\dots,x_n).
\]
Then, we deduce from the differentiability of $f_n$ on $\bigcup_j A_j$ that the function $f_{i,n}^{(\ell)}$ is differentiable at any point $x_i^{(\ell)}\in \R$, such that
$\|x_i\|\neq \|x_j\|$ for all $j \neq i$.
Since in addition $f_{i,n}^{(\ell)}$ is continuous on $\R$, we can apply Lemma~\ref{lem:jump} to deduce that $f_{i,n}^{(\ell)}$ admits a weak derivative defined as
\[
  [f_{i,n}^{(\ell)}]'(x_i^{(\ell)})=\frac{\partial}{\partial x_i^{(\ell)}}[\Psi(\|x_i\|^2)]\1_{A_i}(x_1,\dots,x_n)=2 x_i^{(\ell)} \Psi'(\|x_i\|^2)\1_{A_i}(x_1,\dots,x_n)
\]
which also means that for any $(i,\ell)$, $f_n$ admits a weak partial derivative w.r.t. $x_{i}^{(\ell)}$ defined by
\begin{equation}\label{e:deriv:interm}
  \frac{\partial f_n}{\partial x_i^{(\ell)}}=\frac{\partial}{\partial x_i^{(\ell)}}[\Psi(\|x_i\|^2)]\1_{A_i}(x_1,\dots,x_n)=2 x_i^{(\ell)} \Psi'(\|x_i\|^2)\1_{A_i}(x_1,\dots,x_n).
\end{equation}
We also deduce that
\[
  \sum_{\ell=1}^d\frac{\partial f_n}{\partial x_i^{(\ell)}}x_i^{\ell}= 2 \sum_{\ell=1}^d[x_i^{(\ell)}]^2 \Psi'(\|x_i\|^2)\1_{A_i}(x_1,\dots,x_n)=2\|x_i\|^2 \Psi'(\|x_i\|^2)\1_{A_i}(x_1,\dots,x_n),
\]
which, once combined with the definition of $\nabla_{x_i} f_n$, yields
\begin{equation}\label{e:deriv:interm:2}
  \sum_{i=1}^n\nabla_{x_i}f_n(x_1,\dots,x_n)\cdot x_i=\sum_{i=1}^n\sum_{\ell=1}^d\frac{\partial f_n}{\partial x_i^{(\ell)}}x_i^{(\ell)}= 2\sum_{i=1}^n\|x_i\|^2 \Psi'(\|x_i\|^2)\1_{A_i}(x_1,\dots,x_n).
\end{equation}
For any $i$ and any $(x_1,\dots,x_n)\in A_i$, $\|x_{(k),n}\|^2\Psi'(\|x_{(k),n}\|^2)= \|x_i\|^2\Psi'(\|x_i\|^2)$. So on $\cup_j A_j$ (which is a set of full measure)
\[
  \|x_{(k),n}\|^2\Psi'(\|x_{(k),n}\|^2)=\sum_{i=1}^n \|x_i\|^2\Psi'(\|x_i\|^2)\1_{A_i}(x_1,\dots,x_n).
\]
Equation~\eqref{e:equality:distrib} of Lemma~\ref{lem:deriv} follows from the last equation and from~\eqref{e:deriv:interm:2}. Furthermore, if $\Psi$ is compactly supported in $[-\sqrt{a},\sqrt{a}]$ then $f_n$ coincides  with
\linebreak $\widetilde{f}_n(x_1,\dots,x_n)=\Psi(\|x_{(k),n}\|)u(x_1,\dots,x_n)$ on $(\overline{B}(0,a))^n$
where $u$ is a compactly supported and infinitely differentiable function such that $u\equiv 1$ on $(\overline{B}(0,a))^n$ and $u\equiv 0$ on $(\R^d)^n\setminus(\overline{B}(0,2a))^n$. Using the smoothness of $u$ and the continuity of $f_n$, we deduce that $\widetilde{f}_n$ is also continuous on $(\R^d)^n$. Then, we get in
particular that $\widetilde{f}_n\in L^2((\R^d)^n,\R)$. In addition, $\Psi'$ is also compactly
supported and by the smoothness of $u$, we have for any $i,\ell$
\[
  \frac{\partial \widetilde{f}_n}{\partial x_i^{(\ell)}}=u\frac{\partial f_n}{\partial x_i^{(\ell)}}+ f_n\frac{\partial u}{\partial x_i^{(\ell)}}
\]
in the sense of weak partial derivatives. Using once more~\eqref{e:deriv:interm}, we deduce that for any
$(i,\ell)$, $\frac{\partial \widetilde{f}_n}{\partial x_i^{(\ell)}}$ also belongs to
$L^2((\R^d)^n)$. Hence $\widetilde{f}_n\in H^1((\R^d)^n)$.

Since on $(\overline{B}(0,a))^n$ $f_n$ coincides with the function $\widetilde{f}_n$, which belongs to \linebreak$H^1((\R^d)^n, \R)$,~we deduce that $f_n\in H^1((\overline{B}(0,a))^n,\R)$. The fact that $\sup_n\|f_n\|_{L^\infty(\R^n)}+\sum_i\|\nabla_{x_i}f_n\|_{L^\infty(\R^n)}<\infty$ also comes from~\eqref{e:deriv:interm}.
\end{proof}
\subsection{A density result for the form functions considered in Lemma~\ref{lem:deriv:closed}}
In the following lemma, we state a useful density result to approximate the form functions $f_n$ defined in Lemma~\ref{lem:deriv:closed}.
\begin{lemma}\label{lem:density:Psi}
Let $k\geq 1$ be a fixed integer and let $\Psi$ be a function belonging to $H^1([0,1],\R)$. Define $\Psi_n(x_1,\dots,x_n)=\Psi(1) \1(n<k) + \Psi(\|x_{(k),n}\|^2) \1(n\geq k)$.
Then, there exists a sequence $(f_{\ell,n})$ of symmetric functions of $\mathcal{C}^1(W^n,\R)$ such that\\
(i) For any $n,\ell\geq 0$, ${f_{\ell,n+1}}_{\big|x_{n+1}\in \partial W}\equiv f_{\ell,n}$.\\
(ii)There exists some $C>1$ such that for any $n,\ell\geq 0$, $\|f_{\ell,n}- \Psi_n\|_{H^1(W^n,\R)}\leq C^n/\ell$.
(iii)For each $\ell$, there exists some $C'_\ell>1$ such that for any $n\geq 0$
\[
\|f_{\ell,n}\|_{L^\infty(W^n,\R)}+\sum_i\|\nabla_{x_i}f_{\ell,n}\|_{L^\infty(W^n,\R)}\leq (C'_\ell)^n
\]
\end{lemma}
\begin{remark} The notation $\Psi(1)$ makes sense if $\Psi\in H^1([0,1],\R)$ since
the usual Sobolev injection yields that $H^1([0,1],\R)\hookrightarrow
\mathcal{C}^0([0,1],\R)$.
\end{remark} 
\begin{proof}
The case $d=1$ would deserve a particular treatment, but as it can be easily
adapted from the case $d \ge 2$, we only handle the latter one. For any $n\geq 1$, and any $(r_1,\dots,r_n)\in (0,1)^n$ define by
induction
\begin{align*}
r_{(1),n}&=\min(r_1,\dots,r_n)\\
r_{(k),n}&=\min\{r\in \{r_1,\dots,r_n\}\setminus\{r_{(1),n},\dots,r_{(k-1),n}\}\}\;\mbox{ for }2\leq k\leq n.
\end{align*}
Then fix $k\geq 1$ and define the following functions $\Phi_n$ on $[0,1]^n$ for any $n\geq 0$ by $\Phi_n(r_1,\dots,r_n)=
  ( \Psi(r_{(k),n})-\Psi(1) ) \; \1(n\geq k)$. In particular $\Phi_0\equiv 0$. We observe that
\[
\Psi_n(x_1,\dots,x_n)=\Phi_n(\|x_1\|^2,\dots,\|x_n\|^2)+\Psi(1)\mbox{ on }W^n.
\]
If we approximate $\Phi_n$ by density using
Proposition~\ref{pro:density:compatibility},
Lemma~\ref{lem:density:Psi} will be deduced using a change of variables in polar coordinates justified in the sequel.

Since Proposition~\ref{pro:density:compatibility} applies to a sequence of functions defined on $\R^n$, we need at first to extend each function $\Phi_n$ on $\R^n$ into a function $\widetilde{\Phi}_n$ belonging to $H^1(\R^n,\R)$ and satisfying $\|\widetilde{\Phi}_n\|_{H^1(\R^n,\R)}\leq B^n$ for some $B>1$. This will be possible since for any $n$, $\Phi_n\in H^1((0,1)^n,\R)$. Hereafter, we prove that the new sequence $(\widetilde{\Phi}_n)$ satisfies the assumptions (a)--(c) of
Proposition~\ref{pro:density:compatibility} with $a=1$.

\noindent{\it Step 1: we prove that for any $n$, $\Phi_n\in H^1((0,1)^n,\R)$}. We focus on the case $n\geq k$ as the case $n < k$ is obvious. Since $\Psi\in
H^1((0,1),\R)\subset L^2((0,1),\R)$
\begin{align*}
\int_{(0,1)^n}|\Phi_n(r_1,\dots,r_n)|^2 \dd r_1\dots \dd r_n&=\int_{(0,1)^n}|\Psi(r_{(k),n})-\Psi(1)|^2 \dd r_1\dots \dd r_n\\
&=\sum_{i=1}^n\int_{\begin{subarray}{c}
(r_1,\dots,r_n)\in
  (0,1)^n \\
  r_i=r_{(k),n}
  \end{subarray}}|\Psi(r_i)-\Psi(1)|^2\dd r_1\dots \dd r_n\\
&\leq 2\left[\sum_{i=1}^n\int_{0}^1(|\Psi(r_i)|^2+|\Psi(1)|^2)
  \dd r_i\right]\\
&\quad \times \int_{(r_1,\dots,r_{i-1}, r_{i+1},\dots r_n)\in
    (0,1)^{n-1}}\hspace*{-1cm}\dd r_1\dots \dd r_{i-1} \dd r_{i+1}\dots \dd r_n \\
  &\leq  2n(\|\Psi\|_{L^2((0,1),\R)}^2 + |\Psi(1)|^2)<\infty
\end{align*}
whereby we deduce that $\Phi_n\in L^2((0,1)^n,\R)$ for any $n\geq 0$.

To deduce that each $\Phi_n$ belongs to $H^1((0,1)^n,\R)$, we prove that for any $n\geq 0$, $\Phi_n\in L^2((0,1)^n,\R)$ admits partial
derivatives w.r.t. $r_i$ for any $i$ and that these partial derivatives are all square integrable. Since $\Psi\in H^1((0,1),\R)$, $\Psi$ is continuous on $[0,1]$ as well as $(r_1,\dots,r_n)\mapsto r_{(k),n}$. Hence, $\Phi_n$ is continuous on $W^n$.
By applying Lemma~\ref{lem:jump} to the functions $\Phi_n^{(i)}:r_i\mapsto \Phi_n(r_1,\dots,r_i,\dots,r_n)$, we deduce that for any $(n,i)$, $\Phi_n$ admits a weak derivative w.r.t. $r_i$ and that for a.e. $(r_1,\dots,r_n)$
\[
\frac{\partial \Phi_n(r_1,\dots,r_n)}{\partial r_i}=2r_{(k),n}\Psi'(r_{(k),n})\1_{r_{(k),n}=r_i}(r_1,\dots,r_n).
\]
Since $\Psi\in H^1((0,1),\R)$ we have for  any $i\in\{1,\dots,n\}$
\begin{align*}
\int_{(0,1)^n}\left|\frac{\partial \Phi_n(r_1,\dots,r_n)}{\partial
  r_i}\right|^2 \dd r_1\dots \dd r_n&\leq
\sum_{i=1}^n\int_{0}^1\left|\frac{\partial \Psi(r_i)}{\partial
  r_i}\right|^2\times 2 r_i \dd r_i\\
&\leq 2\sum_{i=1}^n\int_{0}^1\left|\frac{\partial \Psi(r_i)}{\partial
  r_i}\right|^2 \dd r_i<\infty,
\end{align*}
which shows that $\Phi_n\in H^1((0,1)^n,\R)$ and satisfies $\|\Phi_n\|_{H^1((0,1)^n,\R)}\leq nA$ for some $A>0$\\

\noindent{\it Step~2: extension of $\widetilde{\Phi}_n$ and properties.} Let us define the sequence $\widetilde{\Phi}_n$ as follows
\[
\widetilde{\Phi}_n=
\begin{cases}
  \widetilde{\Phi}_n(r_1,\dots,r_n)=\Phi_n(|r_1|,\dots,|r_n|) & \mbox{ if }(r_1,\dots,r_n)\in (-1,1)^n\\
  \widetilde{\Phi}_n(r_1,\dots,r_n)=0 & \mbox{ otherwise.}
\end{cases}
\]
We prove that this sequence satisfies the assumptions of Proposition~\ref{pro:density:compatibility}. First, Assumption (a) is clearly satisfied. Second, since by Step~1, $\Phi_n\in H^1((0,1)^n,\R)$ for any $n$ and is symmetric, it is also clear from the definition of the sequence $\widetilde{\Phi}_n$ that these functions all belong to $H^1(\R^n,\R)$ with $\|\widetilde{\Phi}_n\|_{H^1(\R^n,\R)}=\|\Phi_n\|_{H^1((0,1)^n,\R)}$ and are also symmetric. This proves Assumption (b). Third, we can check that Assumption (c) is satisfied for the functions $(\Phi_n)$. If $r_{n+1}=1$ and $(r_1,\dots,r_n)\in
(0,1)^n$, $r_{(k),n+1}=r_{(k),n}$, which implies that for any $n$ ${\Phi_{n+1}}_{\big| r_{n+1}=1}\equiv \Phi_n$. The case  $n<k-1$ is also clear. When $n=k-1$, $r_{k}=1$ so $r_{(k),k}=1$ and $\Psi(r_{(k),k})-\Psi(1)=0$. Then, we deduce that Assumption (c) is satisfied by the sequence $(\Phi_n)$. By definition of the sequence $\widetilde{\Phi}_n$, it is obvious that Assumption (c) is also satisfied for the functions $\widetilde{\Phi}_n$.\\

\noindent{\it Step 3: construction of the sequences $(f_{\ell,n})$ to approximate the functions $\Phi_n$}. By Step 2, Proposition~\ref{pro:density:compatibility} can be applied to the sequence
$(\widetilde{\Phi}_n)$ with $a=1$. We deduce
the existence of a sequence $g_{\ell,n}$ of symmetric functions on $\R^n$ satisfying (A)-(D) of Proposition~\ref{pro:density:compatibility}. To conclude, we define the sequence of symmetric functions $f_{\ell,n}$ on $W^n$ by $f_{\ell,n}(x_1,\dots,x_n)=g_{\ell,n}(\|x_1\|^2,\dots,\|x_n\|^2)+\Psi(1)$
and check that the desired result holds. The compatibility relations are
clearly satisfied. Furthermore, since for any $(\ell,n)$,
$g_{\ell,n}\in\mathcal{C}^1(\R^n,\R)$, then
$f_{\ell,n}\in\mathcal{C}^1(W^n,\R)$. In addition
\begin{align*}
\|f_{\ell,n}\|_{L^\infty(W^n,\R)}&=\|g_{\ell,n}\|_{L^\infty((0,1)^n,\R)},\\
\|\nabla_{x_i}f_{\ell,n}\|_{L^\infty(W^n,\R)}&=\|2r_i\partial_i g_{\ell,n}\|_{L^\infty((0,1)^n,\R)}\leq 2\|\partial_i g_{\ell,n}\|_{L^\infty((0,1)^n,\R)}.
\end{align*}
Then, we easily deduce (iii) from property (B) satisfied by the sequence $g_{\ell,n}$.

To achieve the proof of Lemma~\ref{lem:density:Psi}, we prove that  $\|f_{\ell,n}- \Psi_n\|_{H^1(W^n,\R)}\leq C^n/\ell$ for some $C>1$. For any $i$ and any $x_i\in W$, let
$r_i=\|x_i\|$, $\Theta_i={x_i}/{\|x_i\|}$, $u_i=r_i^2$ and $\mathbb{S}^{d-1}=\{x\in \R^d,\,\|x\|=1\}$. We have
\begin{align*}
\int_{W^n}|&f_{\ell,n}(x_1,\dots,x_n)-\Psi_n(x_1,\dots,x_n)|^2\,\dd x_1\dots
\dd x_n\\
=&|\mathbb S^{d-1}|^n\left[\int_{(0,1)^n}|g_{\ell,n}(u_1,\dots,u_n)-\Phi_n(u_1,\dots,u_n)|^2\,\prod_{i=1}^n
  (u_i)^{(d-1)/2-1/2}\dd u_1\dots \dd u_n\right].
\end{align*}
Since $d\geq 2$ and $H^1((0,1)^n,\R)\hookrightarrow L^2((0,1)^n,\R)$,
\begin{align*}
\int_{(0,1)^n}|&g_{\ell,n}(u_1,\dots,u_n)-\Phi_n(u_1,\dots,u_n)|^2\,\prod_{i=1}^n
(u_i)^{(d-1)/2-1/2}\dd u_1\dots \dd u_n\\
&=\int_{(0,1)^n}|g_{\ell,n}(u_1,\dots,u_n)-\Phi_n(u_1,\dots,u_n)|^2\prod_{i=1}^n (u_i)^{d/2-1} \dd u_1\dots \dd u_n\\
&\leq \int_{(0,1)^n}|g_{\ell,n}(u_1,\dots,u_n)-\Phi_n(u_1,\dots,u_n)|^2\dd u_1\dots
\dd u_n
\end{align*}
the last display coming from the fact that $d/2-1\geq 0$. It implies that $\|f_{\ell,n}-\Psi_n\|_{L^2(W^n,\R)}\leq |\mathbb S^{d-1}|^n\|g_{\ell,n}-\Phi_n\|_{L^2((0,1)^n,\R)}$. We also note that for any $i$ and any\linebreak $(x_1,\dots,x_n)\in W^n$,
$\nabla_{x_i} f(x_1,\dots,x_n)=2x_i\cdot g(\|x_1\|^2,\dots,\|x_n\|^2)$.
The same change of variables as above yields that
$\|\nabla_{x_i} f_{\ell,n}-  \nabla_{x_i}\Psi_n\|_{L^2(W^n,\R)}\leq |\mathbb S^{d-1}|^n\|\nabla_{r_i} g_{\ell,n}-\nabla_{r_i}\Phi_n\|_{L^2((0,1)^n,\R)}$. The point~(ii) is therefore deduced.
\end{proof}

\subsection{Proof of Lemma~\ref{lem:deriv:closed}}

Following Section~\ref{sec:duality}, we extend the closable operator $\nabla : \dom(D) \to L^2(\Omega)$ to
\[
\dom(\overline{D})=\{F\in L^2(\Omega): \,\exists (F_\ell)\in \dom(D),\,\nabla F_\ell\mbox{ converges in }L^2(\Omega)\}.
\]
Now, we apply Lemma~\ref{lem:density:Psi} and choose a sequence
$(f_{\ell,n})\in \mathcal{C}^1(W^n,\R)$ satisfying the compatibility relations
and point~(ii),(iii) of Lemma~\ref{lem:density:Psi}. Let us
define $(F_\ell)_\ell$ the sequence of elements in $L^2(\Omega)$, which admit $(f_{\ell,n})$ as
form functions. We check that this sequence satisfies the following properties:
(i) For any $\ell$, $F_\ell\in \dom(D)$; (ii) $F_\ell \to F$ in $L^2(\Omega)$;
(iii) $\nabla F_\ell$ converges in $L^2(\Omega)$ to some $G_k$.

The property (i) is deduced from the compatibility relations since \linebreak$(f_{\ell,n}) \in \mathcal{C}^1(W^n,\R)$ and point (iii) of Lemma~\ref{lem:density:Psi} holds. (ii) Since the functions $(f_{\ell,n})$ are given by Lemma~\ref{lem:density:Psi}, we have for any $n,\ell\geq 0$, $\|f_{\ell,n}- \Psi_n\|_{L^2(W^n,\R)}\leq C^n/\ell$ and $\|\nabla_{x_i}f_{\ell,n}- \nabla_{x_i}\Psi_n\|_{L^2(W^n,\R)}\leq C^n/\ell$ for some $C>1$. Since for any $C>1$, $\ell^{-1}\sum_n C^n/n! \to 0$ as $\ell\to \infty$, Lemma~\ref{lem:cv} yields the result.

Using the definition of $\nabla F_\ell$, we also have  that $F_\ell\to F\mbox{ in }L^2(\Omega)$ and
$\nabla F_\ell\to G_k\mbox{ in }L^2(\Omega)$
where
\[
G_k=2\sum_{n\geq k}\1(N(W)=n)\,\sum_{i=1}^n\nabla_{x_i}\Phi_n\cdot x_i.
\]
From Lemma~\ref{lem:1}, we have the explicit expression of $G_k$ depending only on $\Psi$:
\begin{equation}\label{e:cond3}
G_k=2\sum_{n\geq k}\1(N(W)=n)\,\|x_{(k),n}\|^2\Psi'(\|x_{(k),n}\|^2).
\end{equation}
The results (i)-(iii) combined with~\eqref{e:cond3} precisely ensure that $F\in \dom(\overline{D})$ and that
$\overline{\nabla} F=G_k$.


\section{Auxiliary density results}\label{s:density}
\subsection{Proof of Lemma~\ref{lem:density-dom}}\label{s:density-dom}

Observe that since ${\mathcal C}^1(W^n,\R)$ is dense in $L^2(W^n,\R)$, $\mathcal S^\prime$ is also dense in $L^2(\Omega)$ from Lemma~\ref{lem:cv}. Then, let us fix $F\in L^2(\Omega)$ and choose $(F_\ell)_\ell \in \mathcal S^\prime$ such that $F_\ell\to F$ in $L^2(\Omega)$. Denote for any $\ell$, $(f_{\ell,n})$ the form functions of the functional $F_\ell$. Since for any $\ell$ $F_\ell\in \mathcal S^\prime$, there exist some $C_\ell>1$ for each $\ell$ such that for any $n$
\begin{equation}\label{e:fn-ell}
\|f_{\ell,n}\|_{L^\infty(W^n,\R)}+\sum_{i=1}^n\|\nabla_{x_i}f_{\ell,n}\|_{L^\infty(W^n,\R^d)}\leq C_\ell^{n}.
\end{equation}
We modify each form function $f_{\ell,n}$ on $W^{n-1}\times \partial W$ to get new form functions $h_{n,\ell}$ related to some functional $H_\ell\in \mathcal S^\prime$ also converging to $F$ in $L^2(\Omega)$. Since the topology involved in the convergence in $L^2(\Omega)$ is the $L^2(W^n)$ convergence, to do so we will modify the functions $x_n\mapsto f_{\ell,n}(\cdot,x_n)$ in a neighborhood of the boundary of $\partial W$ without changing the convergence properties.

Since $W$ has a $\mathcal{C}^2$ boundary, the function $g:x\in W\mapsto d(x,W^c)$ is also $\mathcal{C}^2$ for $d(x,W^c)\leq \varepsilon_0$ for some $\varepsilon_0>0$. In particular, for some $M>1$
\begin{equation}\label{e:dist}
\|g\|_{L^\infty(W,\R)}+\|\nabla g\|_{L^\infty(W,\R^d)}\leq M.
\end{equation}
Now, we define the sequences $(h_{\ell,n})_\ell$ by induction on $n\geq 0$. At first, we set $h_{\ell,0}=f_{\ell,0}$. Then, for each $\ell$, consider $\varepsilon_{\ell,1}\in (0,\varepsilon_0)$ such that
\begin{equation}\label{e:fn-ell-L2}
\int_{ \{z\in W:d(z,W^c)\leq \varepsilon_{\ell,1}\} }[\left|h_{\ell,0}\right|^2+\left|f_{\ell,1}(z)\right|^2\,]\dd z\leq \frac{1}{2\ell^2}\left(\int_{z\in W}\left|f_{\ell,1}(z)\right|^2 \dd z\right)
\end{equation}
and define $h_{\ell,1}$ on $W$ by
\[
h_{\ell,1}(z)=
\begin{cases}
 f_{\ell,1}(z) & \mbox{ if }d(z,W^c)\geq \varepsilon_{\ell,1}\\
\frac{d(z,W^c)}{\varepsilon_{\ell,1}}f_{\ell,1}(z)+(1-\frac{d(z,W^c)}{\varepsilon_{\ell,1}})h_{\ell,0}& \mbox{ otherwise.}
\end{cases}
\]
Since $d(z,W^c)=0$ for $z\in \partial W$, we have  $h_{\ell,1}|_{\partial W}\equiv h_{\ell,0}$ for each $\ell$. Using~\eqref{e:fn-ell-L2}, we also check that $\|h_{\ell,1}-f_{\ell,1}\|_{L^2(W,\R)}\leq \|f_{\ell,1}\|_{L^2(W,\R)}/\ell$. Furthermore, \eqref{e:dist} and \eqref{e:fn-ell} imply that for any $\ell$
\[
\|h_{\ell,1}\|_{L^\infty(W,\R)}+\sum_{i=1}^n\|\nabla_{x_i}h_{\ell,1}\|_{L^\infty(W,\R^d)}\leq MC_\ell.
\]
Assume that we have defined the sequences of functions $(h_{\ell,1}),\cdots,(h_{\ell,N})$ such that
\begin{itemize}
\item (H1) for each $\ell\geq 1$ and $0\leq n\leq N-1$, $h_{\ell,n+1}\equiv h_{\ell,n}$ if $z_{n+1}\in \partial W$.
\item (H2) for each $\ell\geq 1$ and $0\leq n\leq N$, $\|h_{\ell,n}-f_{\ell,n}\|_{L^2(W^n,\R)}\leq \|f_{\ell,n}\|_{L^2(W^n,\R)}/\ell$.
\item (H3) for each $\ell\geq 1$ and $0\leq n\leq N$, $\|h_{\ell,n}\|_{L^\infty(W^n,\R)}+\sum_{i=1}^n\|\nabla_{x_i}h_{\ell,n}\|_{L^\infty(W^n,\R^d)}\leq (MC_\ell)^n$.
\end{itemize}
Let $\Gamma_{\ell,N+1}$ denote the set $\Gamma_{\ell,N+1}=\{(z_1,\cdots,z_{N+1})\in W^{N+1}:d(z_{N+1},W^c)\leq \varepsilon_{\ell,N+1}\}$
To define the sequence $(h_{\ell,N+1})$, we consider, for each $\ell$, some $\varepsilon_{\ell,N+1}\in (0,\varepsilon_0)$ such that
\begin{align}
\int_{\Gamma_{\ell,N+1}} & [\left|h_{\ell,N}(z_1,\cdots,z_{N})\right|^2+\left|f_{\ell,N+1}(z_1,\cdots,z_{N+1})\right|^2]\,\dd z_1\cdots \dd z_{N+1}\nonumber\\
&\leq \frac{1}{2\ell^2}\int_{W^{N+1}}\left|f_{\ell,N+1}(z_1,\cdots,z_{N+1})\right|^2\,\dd z_1\cdots \dd z_{N+1}.\label{e:P-eps}
\end{align}
Then, we define  $h_{\ell,N+1}$ on $W^{N+1}$ by
\begin{align*}
  &h_{\ell,N+1}(z_1,\cdots,z_{N+1})= \\
  &\begin{cases}
    f_{\ell,N+1}(z_1,\cdots,z_{N+1}) &\mbox{on } \Gamma_{\ell,N+1}\\
    \frac{d(z_{N+1},W^c)}{\varepsilon_{\ell,N+1}}f_{\ell,N+1}(z_1,\cdots,z_{N+1})+(1-\frac{d(z_{N+1},W^c)}{\varepsilon_{\ell,N+1}})h_{\ell,N}(z_1,\cdots,z_{N}) &\mbox{ otherwise}
  \end{cases}
\end{align*}
and check that (H1)--(H3) hold. By construction (H1) is valid. Regarding~(H2), from~\eqref{e:P-eps}
\begin{align*}
&\|h_{\ell,N+1}-f_{\ell,N+1}\|_{L^2(W^{N+1})}^2=\int_{\Gamma_{\ell,N+1}}|h_{\ell,N+1}-f_{\ell,N+1}|^2 \dd z_1\cdots \dd z_{N+1}\\
&\leq \int_{\Gamma_{\ell,N+1}}\!\!\left(1-\frac{d(z_{N+1},W^c)}{\varepsilon_{\ell,N+1}}\right)^2|h_{\ell,N}(z_1,\cdots,z_N)-f_{\ell,N+1}(z_1,\cdots,z_{N+1})|^2\dd z_1\cdots \dd z_{N+1}\\
&\leq 2\int_{\Gamma_{\ell,N+1}}|f_{\ell,N+1}(z_1,\cdots,z_{N+1})|^2 \dd z_1\cdots \dd z_{N+1}\\
&+2\int_{\Gamma_{\ell,N+1}}\left(1-\frac{d(z_{N+1},W^c)}{\varepsilon_{\ell,N+1}}\right)|h_{\ell,N}(z_1,\cdots,z_N)|^2 \dd z_1\cdots \dd z_{N+1}\\
&\leq \frac{1}{\ell^2}\left(\int_{W^{N+1}}\left|f_{\ell,N+1}(z_1,\cdots,z_{N+1})\right|^2\, \dd z_1\cdots \dd z_{N+1}\right).\\
\end{align*}
To check (H3), note that from \eqref{e:fn-ell} and (H3) we have
\begin{align*}
\|h_{\ell,N+1}\|_{L^\infty(W^{N+1},\R)}&+\sum_{i=1}^n\|\nabla_{x_i}h_{\ell,N+1}\|_{L^\infty(W^{N+1},\R^d)}\\
&\leq \max\left(\|f_{\ell,N+1}\|_{L^\infty(W^{N+1},\R)}+\sum_i\|\nabla_{x_i}f_{\ell,N+1}\|_{L^\infty(W^{N+1},\R^d)} \right. \\
  & \left. \phantom{\leq \max} \quad \vphantom{\leq \max\|f_{\ell,N+1}\|_{L^\infty(W^{N+1},\R)}+\sum_{i=1}^n\|\nabla_{x_i}f_{\ell,N+1}\|_{L^\infty(W^{N+1},\R^d)}},
  M(\|h_{\ell,N}\|_{L^\infty(W^{N},\R)}+
\sum_i\|\nabla_{x_i}h_{\ell,N}\|_{L^\infty(W^{N},\R^d)})\right)\\
&\leq (MC_\ell)^{N+1}.
\end{align*}

To conclude, let $H_\ell$ be the functional with form functions $h_{\ell,n}$. Since (H1) and (H3) are satisfied, $H_\ell\in\dom(D^\pi)$ for each $\ell$. Finally from (H2), we obtain that $\|H_\ell-F_\ell\|_{L^2(\Omega)}\leq \|F_\ell\|_{L^2(\Omega)}/\ell$, which combined with the convergence of $F_\ell$ to $F$ yields that $H_\ell\to F$ in $L^2(\Omega)$.

\subsection{Density results used in the proof of Lemma~\ref{lem:deriv:closed}}
In this section, we state some  density results used in the
  proof of Lemma~\ref{lem:deriv:closed} and more precisely in the proof of Lemma~\ref{lem:density:Psi}.
Let $p\in (1,+\infty)$ and $p'$ such that $1/p+1/p'=1$. We introduce the so--called Bessel--potential spaces defined for $s>0$  by
\[
H^s_p(\R^n,\R)=\{u\in \mathcal{S}'(\R^n,\R),\,(1+|\xi|^2)^{s/2}\widehat{u}\in L^{p'}(\R^n,\R)\}
\]
endowed with the norm $\|u\|_{H^s_p(\R^n,\R)}=\|(1+|\xi|^2)^{s/2}\widehat{u}\|_{L^{p'}(\R^n,\R)}$. When $p=2$, we recover the usual Sobolev spaces $H^s(\R^n,\R)$.
We also introduce the integrals
\begin{equation}\label{e:cm}
\Lambda_m=\left(\int_\R (1+\xi^2)^{-m}\dd\xi\right)^{1/2}\;
\end{equation}
and denote by $\mathbb B^n$ the Euclidean ball in dimension $n$ with radius 1 for which we recall that $|\mathbb B^n|=\pi^{n/2}/\Gamma(n/2+1)\leq 5.3$.
\begin{lemma}\label{lem:lambda-m}
The sequence $(\Lambda_m)$ is non--increasing for $m>1/2$ and for any $m\geq 1$ and  $n\geq 1$
\begin{equation}\label{e:lambda-m}
\int_{\R^n}\frac{1}{(1+|\xi|^2)^{m}}\dd\xi\leq |\mathbb B^n|\Lambda_{(m-(n-1))/2}^2.
\end{equation}
\end{lemma}
\begin{proof}
By a change of variable in polar coordinates we get that
\[
\int_{\R^n}\frac{1}{(1+|\xi|^2)^{m}}\dd\xi=|\mathbb B^n|\int_{\R}\frac{r^{n-1}\dd r}{(1+r^2)^{m}}\leq |\mathbb B^n|\Lambda_{(m-(n-1))/2}^2.
\]
\end{proof}
Now, we make precise a result established by \citet[Section 3.3.1]{trie:83}.

\begin{lemma}\label{lem:emb-C1}
Let $p\in (1,+\infty)$, $\varepsilon>0$ and $s>1+\varepsilon+n/p$. For any $u\in H^s_p(\R^n,\R)$,  $u\in\mathcal{C}^1(\R^n,\R)$ and 
for some $C>1$ depending only on $\varepsilon$
\begin{equation}\label{e:emb}
\|u\|_{L^\infty(\R^n,\R)}+\|\nabla u\|_{L^\infty(\R^n,\R)}\leq C\|u\|_{H^s_p(\R^n,\R)}.
\end{equation}
\end{lemma}
\begin{proof}
We use the density of the Schwartz class in $H^s_p(\R^n,\R)$ and first prove the inequality for $u$ belonging to this class. We use that
\[
u(x)=\int_{\R^n} \widehat{u}(\xi)e^{i\xi \cdot x}\dd\xi=\int_{\R^n}\left[(1+|\xi|^2)^{s/2}\widehat{u}(\xi)\right]\frac{e^{i\xi \cdot x}}{(1+|\xi|^2)^{s/2}}\dd\xi.
\]
From Cauchy Schwartz inequality 
\begin{equation}\label{e:emb-1}
|u(x)|\leq \left(\int_{\R^n}\frac{1}{(1+|\xi|^2)^{sp/2}}\dd\xi\right)^{1/p}\|u\|_{H^s_p(\R^n,\R)}\leq |\mathbb B^n|^{1/p}\Lambda_{sp/2}^{2/p}\|u\|_{H^s_p(\R^n,\R)}.
\end{equation}
Let $\ell\in\{1,\cdots,n\}$,
\[
\frac{\partial u}{\partial x_\ell}=\int_{\R^n} \xi_\ell \widehat{u}(\xi)e^{i\xi\cdot x}\dd\xi=\int_{\R^n}\left[(1+|\xi|^2)^{s/2}\widehat{u}(\xi)\right]\frac{\xi_\ell e^{i\xi\cdot x}}{(1+|\xi|^2)^{s/2}}\dd\xi.
\]
Then, from Lemma~\ref{lem:lambda-m} we get
\begin{align}\label{e:emb-2}
\left|\frac{\partial u}{\partial x_\ell}\right|&\leq \|u\|_{H^s_p(\R^n,\R)} \int_{\R^n}\frac{|\xi_\ell|^p}{(1+|\xi|^2)^{sp/2}}\dd\xi \leq |\mathbb B^n|^{1/p}\Lambda_{((s-1)p-(n-1))/2}^{2/p}\|u\|_{H^s_p(\R^n,\R)}.
\end{align}
By combining~\eqref{e:emb-1},~\eqref{e:emb-2} and since the function $m\mapsto \Lambda_m$ is non--increasing, we deduce that
 \[
 \|u\|_{L^\infty(\R^n,\R)}+\|\nabla u\|_{L^\infty(\R^n,\R)}\leq 2|\mathbb B^n|^{1/p}\Lambda_{((s-1)p-(n-1))/2}^{2/p}\|u\|_{H^s_p(\R^n,\R)}.
 \]
Since
\[
\sup_{n}\sup_{s>1+\varepsilon+n/p}\sup_{p>1}2|\mathbb B^n|^{1/p}\Lambda_{((s-1)p-(n-1))/2}^{2/p}\leq 2 \sup_{n}|\mathbb B^n|^{1/p}\Lambda_{(1+\varepsilon)/2}^{2}<\infty.
\]
\eqref{e:emb} is deduced for any function $u$ of the Schwartz class which leads to the result since this class is dense in $H^s_p(\R^n,\R)$.
\end{proof}

Now, we recall some basic properties of the trace operator, see~\cite{adams:75}.
\begin{lemma}
Fix $a\in\R$ and $\sigma>1/p$. The mapping
$
\gamma_{n,a}:\varphi\in \mathcal{D}(\R^{n+1},\R)\subset H^\sigma_p(\R^{n+1},\R)\mapsto
\varphi_{\big| x_{n+1}=0}=\int_{\R^{n+1}} e^{i x'\xi'}\widehat{\varphi}(\xi_1,\cdots,\xi_n,\xi_{n+1})d\xi'd\xi_{n+1}\in H^{\sigma-1/p}_p(\R^{n},\R)
$
is continuous with norm less than $\Lambda_\sigma^{2/p}$.
\end{lemma}
The {trace} operator $\gamma_{n,a}^{(p)}$ is defined as its continuous extension from $H^m_p(\R^{n+1},\R)$ to $H^{m-1/p}_p(\R^{n},\R)$. Since it is defined by density on $\mathcal{D}(\R^{n+1},\R)\subset\bigcap_{p,s>1/p} H^s_p(\R^{n+1},\R)$, for $p_1\neq p_2$ the two operators $\gamma_{n,a}^{(p_1)}$ and $\gamma_{n,a}^{(p_2)}$ coincide. For the sake of simplicity, we drop the index $p$, and always denote by $\gamma_{n,a}f$ the trace of a function $f$ belonging to some space $H^s_p(\R^{n+1},\R)$ for some $s>1/p$ and $p>1$.

In addition, we give an explicit expression of the extension operator to the Schwartz class, see \citet[Lemma~1.17, Chapter 11]{zuily:02}). The following result is a slight modification adapted to our framework.

\begin{lemma}\label{lem:explicit:book}
Let $a\in\R$, $M>0$, $p\geq 2$, $s\in (0,2M)$ and $g$ belonging to the Schwartz class. Then, the function $f$ defined from $g$ on $\R^{n+1}$ in the Fourier domain by
\begin{equation}\label{e:explicit:ext}
\widehat
f(\xi',\xi_{n+1})=\Lambda_{M+1/2}^{-2}\frac{(1+|\xi'|^2)^M\,\widehat{g}(\xi')}{(1+|\xi'|^2+\xi_{n+1}^2)^{M+1/2}},\; \forall \xi=(\xi',\xi_{n+1})\in \R^{n+1}
\end{equation}
is an element of  $H^{s+1/p}(\R^{n+1},\R)$ and satisfies
\begin{equation}\label{e:explicit:norm-bessel}
\|f\|_{H^{s+1/p}_p(\R^{n+1},\R)}\leq \Lambda_{M+1/2}^{-2}\Lambda_{Mp+p/2-1/2+sp/2}^{2/p}\|g\|_{H^{s}_p(\R^{n},\R)}
\end{equation}
and $\gamma_{n,a} f=g$. Moreover, if $g$ is symmetric, so is $f$.
\end{lemma}
\begin{remark}
The function $f$ is not unique (it depends on $M$) since the trace operator is only a surjective operator. The main point is that we can extend the function $g$ so that inequalities~\ref{e:explicit:norm-bessel} are valid for a large range of values of $(p,s)$ and for the {\it same functions} $f$.
\end{remark}
\begin{proof}
All the conclusions of Lemma~\ref{lem:explicit:book} are stated and proved in~\citet[Lemma~1.17 of Chapter 11]{zuily:02}, except~\eqref{e:explicit:norm-bessel}, which we now focus on.
To prove~\eqref{e:explicit:norm-bessel}, observe that since $p\geq 2$
\begin{align*}
(1+|\xi|^2)^{(s+1/p)p/2}|\widehat{f}(\xi)|^p&=\Lambda_{M+1/2}^{-2p}(1+|\xi'|^2)^{Mp}\frac{|\widehat{g}(\xi')|^p}{(1+|\xi|^2)^{Mp+p/2-sp/2-1/2}}\\
&\leq \Lambda_{M+1/2}^{-2p}(1+|\xi'|^2)^{Mp}\frac{|\widehat{g}(\xi')|^p}{(1+|\xi|^2)^{Mp+1/2-sp/2}}.
\end{align*}
Since $M>s/2$, $2Mp+1-sp>1$ which implies $\xi_{n+1}\mapsto (1+|\xi|^2)^{-(Mp+1/2-sp/2)}$ is integrable. Hence
\begin{align*}
\int_{\R}(1+|\xi|^2)^{(sp+1)/2}|\widehat{f}(\xi)|^p\,\dd\xi_{n+1}=&\Lambda_{M+1/2}^{-2p}\left(\int_\R (1+|\xi|^2)^{-(Mp+1/2-sp/2)}\dd\xi_{n+1}\right) \\
&\;\times (1+|\xi'|^2)^{Mp}|\widehat{g}(\xi')|^p.
\end{align*}
Denote $\xi_{n+1}=(1+|\xi'|^2)^{1/2}\eta$, then
\begin{align*}
\int_\R (1+|\xi|^2)^{-(Mp+1/2-sp/2)}\dd\xi_{n+1}&=\frac{(1+|\xi'|^2)^{1/2}}{(1+|\xi'|^2)^{Mp+1/2-sp/2}} \\
& \quad \times \int_\R (1+|\eta|^2)^{-(Mp+1/2-sp/2)}\dd\eta\\
&=\Lambda_{Mp+1/2-sp/2}^2\frac{(1+|\xi'|^2)^{1/2}}{(1+|\xi'|^2)^{Mp+1/2-sp/2}}
\end{align*}
whereby we deduce that
\begin{align*}
\int_{\R^{n+1}}&(1+|\xi|^2)^{(sp+1)/2}|\widehat{f}(\xi)|^p\,\dd\xi_{n+1}\\
&=\Lambda_{M+1/2}^{-2p}\Lambda_{Mp+1/2-sp/2}^2\int_{\R^n}(1+|\xi'|^2)^{Mp}\cdot \frac{(1+|\xi'|^2)^{1/2}}{(1+|\xi'|^2)^{Mp+1/2-sp/2}}|\widehat{g}(\xi')|^p \dd\xi'\\
&=\Lambda_{M+1/2}^{-2p}\Lambda_{Mp+1/2-sp/2}^2\int_{\R^n}(1+|\xi'|^2)^{sp/2}|\widehat{g}(\xi')|^p \dd\xi'\\
&=\Lambda_{M+1/2}^{-2p}\Lambda_{Mp+1/2-sp/2}^2\|g\|_{H^{s}_{p}(\R^{n},\R)}^p.
\end{align*}
\end{proof}

Based on the previous lemma, we can state the following one.
\begin{lemma}\label{lem:explicit}
  Let $a\in \R$, $M>0$, $p\geq 2$, $s\in (0,2M)$ and $g\in H^s_p(\R^n,\R)$. Then, there exists some function $f$ belonging to $H^{s+1/p}_p(\R^{n+1},\R)$ such that $\gamma_{n,a} f=g$. In addition,
\begin{equation}\label{e:bound:extend}
\|f\|_{H^{s+1/p}_p(\R^{n+1},\R)}\leq K_M\|g\|_{H^{s}_p(\R^{n},\R)}
\end{equation}
where $K_M>0$ depends only on $M$.
Moreover, if $g$ is symmetric, $f$ can also be chosen symmetric.
\end{lemma}
\begin{proof}
  We consider a sequence $(g_\ell)$ of functions of belonging to the Schwartz calss converging to $g$ in $H^s_p(\R^n,\R)$. We fix some $M>s/2$. For any $\ell$, we define an extension $(f_\ell)$ of the function $(g_\ell)$ using~\eqref{e:explicit:ext}. Since $(g_\ell)$ is a Cauchy sequence in $H^s_p(\R^n,\R)$, we deduce from~\eqref{e:explicit:norm-bessel} applied with $\sigma=s$ that
\[
  \|f_\ell-f_m\|_{H^{s+1/p}_p(\R^{n+1},\R)}\leq \Lambda_{M+1/2}^{-2}\Lambda_{Mp+p/2-1/2+sp/2}^{2/p}\|g_\ell-g_m\|_{H^{s}_p(\R^{n},\R)}\to 0
\]
as $\ell,m\to\infty$.
The sequence $(f_\ell)$ is a Cauchy sequence in $H^{s+1/p}_p(\R^{n+1},\R)$ and converges to
some function $f$ in $H^{s+1/p}_p(\R^{n+1},\R)$. This convergence also
holds in $H^{s+1/p}_p(\R^{n+1},\R)$. Hence, we can let
$\ell \to \infty$ in the inequality
\[
\|f_\ell\|_{H^{s+1/p}_p(\R^{n+1},\R)}\leq \Lambda_{M+1/2}^{-2}\Lambda_{Mp+p/2-1/2+sp/2}^{2/p}\|g_\ell\|_{H^{s}_p(\R^{n},\R)}
\]
to get
\[
\|f\|_{H^{s+1/p}_p(\R^{n+1},\R)}\leq \Lambda_{M+1/2}^{-2}\Lambda_{Mp+p/2-1/2+sp/2}^{2/p}\|g\|_{H^{s}_p(\R^{n},\R)}.
\]

To deduce~\eqref{e:bound:extend}, we observe that, since $(\Lambda_m)$ is a non--decreasing sequence, \linebreak$\sup_{p\geq 2}\Lambda_{Mp+p/2-1/2+sp/2}^{2/p}\leq \sup_{p\geq 2}\Lambda_{2M+1/2}^{2/p}<+\infty$. Therefore, we can set
\[
K_M=\Lambda_{M+1/2}^{-2}\left(\sup_{p\geq 2}\Lambda_{Mp+p/2-1/2+sp/2}^{2/p}\right).
\]

If $g$ is symmetric, we can replace the sequence $(g_\ell)$ by its symmetric
part. Therefore, the functions $(f_\ell)$ are symmetric and so is $f$.
\end{proof}
From this result, we can deduce the next proposition used in the proof of
Lemma~\ref{lem:deriv:closed}.

\begin{proposition}\label{pro:density:compatibility}
Let $a\in \R$ and $k\in \N\setminus \{0\}$. Let $(\Phi_n)_{n\geq 0}$ be a sequence satisfying: (a) $\Phi_n\equiv 0$ if $n<k$. (b) For any $n\geq k$, $\Phi_n$ is a symmetric function belonging to $H^1(\R^n,\R)$ such that $\|\Phi_n\|_{H^1(\R^n,\R)}\leq B^n$ for some $B>1$. (c) For any $n\geq 0$, $\gamma_{n,a}\Phi_{n+1}\equiv \Phi_n$.
Then, there exists a sequence $(g_{\ell,n})_{n\geq 0,\ell}$ such that
\begin{itemize}
  \item[(A)] For any $\ell$ and any $n<k$, $g_{\ell,n}=0$.
  \item [(B)] For any $n\geq k$ and $\ell$, $g_{\ell,n}$ are symmetric functions belonging to $\mathcal{C}^1(\R^n,\R)$ 
      such that for any $\ell$, $\|g_{\ell,n}\|_{L^\infty(\R^n,\R)}+\|\nabla g_{\ell,n}\|_{L^\infty(\R^n,\R)}\leq M_\ell^n$ for some $M_\ell>1$.
  \item[(C)] For any $\ell$ and $n$, $\|g_{\ell,n}-\Phi_n\|_{H^1(\R^n,\R)}\leq C^n/\ell$ for some $C>1$ independent on $\ell$ and $n$.
  \item[(D)] For any $\ell$, the sequence $g_{\ell,n}$ also satisfies the relation $\gamma_{n,a} g_{\ell,n+1}\equiv g_{\ell,n}$.
\end{itemize}
\end{proposition}
\begin{proof}
Up to a translation, we can assume that $a=0$. We can also
assume that $k=1$ and we denote $\gamma_n=\gamma_{n,0}$.

\noindent{\it Step~1:  Approximation of functions $\Phi_n$ by smooth functions $(h_{\ell,n})$}.
Set for any $\ell$,
$g_{\ell,0}=h_{\ell,0}= 0$. Let $\ell$ a fixed integer. There exists $n_\ell$ such that for any $n\geq n_\ell$, $B^{n}\geq \ell$. For $n\geq n_\ell$, we set $h_{\ell,n}\equiv 0$. Since $\|\Phi_n\|_{H^1(\R^n,\R)}\leq B^n$, we deduce that for $n\geq n_\ell$,
\begin{equation}\label{e:density-hnl}
\|h_{\ell,n}-\Phi_n\|_{H^1(\R^n,\R)}=\|\Phi_n\|_{H^1(\R^n,\R)}\leq B^n=B^{2n}B^{-n}\leq B^{2n}/\ell
\end{equation}
by definition of $n_\ell$. Hereafter, we define $(h_{\ell,n})$ for $n\leq n_\ell-1$. To do so, we apply classical density results and
consider a sequence $(h_{\ell,n})\in\mathcal{D}(\R^n,\R)$, such that for any $\ell$ inequality~\eqref{e:density-hnl} is satisfied for any $1\leq n\leq n_\ell$.

Since the functions $h_{\ell,n}$ are both smooth and compactly supported and $h_{\ell,n}\equiv 0$ for $n\geq n_\ell$, the constant
\begin{equation}\label{e:cl-hnl}
C_{\ell,p}=\sup_{n\geq 0} \|h_{\ell,n}\|_{H^2_p(\R^n,\R)}
\end{equation}
is finite.

We can replace the
sequence $(h_{\ell,n})$ by the symmetric part of each function and since each
$\Phi_n$ is symmetric, we can assume that all the functions $(h_{\ell,n})$ are
symmetric. We modify by induction the sequence $(h_{\ell,n})$ in order to
define a new sequence $(g_{\ell,n})$ satisfying the compatibility
relations~(D).

\noindent{\it Step 2: Definition of the sequences $(g_{\ell,n})$ for $n\geq 0$}. For $n=0$, we set
$g_{\ell,0} \equiv 0$ for any $\ell\geq 1$. When $n=1$, the continuity of the trace of functions of $H^1(\R,\R)$ yields $|\gamma_1 h_{\ell,1}-\Phi_1(0)|=|h_{\ell,1}(0)-\Phi_0|=|h_{\ell,1}(0)|\leq A/\ell$ for some $A>0$. Let $\varphi$ be a symmetric function belonging to $\mathcal{D}(\R,\R)$ such that $\varphi(0)=1$ and set $g_{\ell,1}=h_{\ell,1}-h_{\ell,1}(0)\varphi$.
This sequence clearly belongs to  $\left(\bigcap_{p>1}H^2_p(\R,\R)\right)\cap H^1(\R,\R)$, and all functions $g_{\ell,1}$ are symmetric. First, we can check that for any $\ell$, $g_{\ell,1}(1)=0=g_{\ell,0}$ and so~(D) is satisfied for $n=0$. Second, since $|h_{\ell,1}(0)|\leq A/\ell$ for some $A>0$ and $\|h_{\ell,1}- \Phi_1\|_{H^1(\R,\R)}\leq B^2/\ell$ as stated in~\eqref{e:density-hnl}, we deduce that $\|g_{\ell,1}- \Phi_1\|_{H^1(\R,\R)}\leq C/\ell$ for some $C>1$ with $C=B^2+A>1$.

Assume that we have defined for any $\ell$ by induction some symmetric functions $g_{\ell,N}$, belonging respectively to $\left(\bigcap_{p>N}H^{2}_p(\R^N,\R)\right)\cap H^{1}(\R^N,\R)$,  satisfying (D) for $n=1,\dots,N$ and for any $\ell$, $p>N$
\begin{equation}\label{e:cv:induction1}
\|g_{\ell,n}\|_{H^2_p(\R^n,\R)}\leq D_{\ell,p}^n
\end{equation}
for some $D_{\ell,p}\geq \max(\Lambda_2^{2/p}C_\ell,C_\ell,2K_2+1,\|g_{\ell,1}\|_{H^2_p(\R^n,\R)})$ and
\begin{equation}\label{e:cv:induction2}
\|g_{\ell,N}-\gamma_N h_{N+1,\ell}\|_{H^1(\R^n,\R)}\leq C^N/\ell
\end{equation}
for some $C>1$.
Now, we define for any $\ell$ the functions $g_{\ell,N+1}$. The function $g_{\ell,N+1}$ will be of the form $g_{\ell,N+1}=h_{\ell,N+1}+\widetilde{h}_{\ell,N+1}$, where $\widetilde{h}_{\ell,N+1}$ is a function depending on $g_{\ell,N}$ and $h_{\ell,N+1}$.

Let us explain how the functions $\widetilde{h}_{\ell,N+1}$ are defined. Since we require that $\gamma_N g_{\ell,N+1}=g_{\ell,N}$, we have necessarily  $\gamma_N \widetilde{h}_{\ell,N+1}=\gamma_N g_{\ell,N+1}-\gamma_{N}h_{\ell,N+1}=g_{\ell,N}-\gamma_{N}h_{\ell,N+1}$. Then to define the function $\widetilde{h}_{\ell,N+1}$, we have to extend the function $g_{\ell,N}-\gamma_{N}h_{\ell,N+1}$ to $\R^{n+1}$, which is possible in view of Lemma~\ref{lem:explicit}. We now use induction assumptions~\eqref{e:cv:induction1},~\eqref{e:cl-hnl}
\begin{align*}
\|g_{\ell,N}-\gamma_{N}h_{\ell,N+1}\|_{H^{2}_p(\R^N,\R)} &\leq \|g_{\ell,N}\|_{H^{2}_p(\R^N,\R)}+\|\gamma_{N}h_{\ell,N+1}\|_{H^{2}_p(\R^N,\R)} \\
&\leq D_{\ell,p}^N+\Lambda_2^{2/p}\|h_{\ell,N+1}\|_{H^{2}_p(\R^{N+1},\R)}\leq D_{\ell,p}^N+\Lambda_2^{2/p}C_\ell\leq 2D_{\ell,p}^N
\end{align*}
where the last inequality comes from the bound $D_{\ell,p}\geq \Lambda_2^{2/p}C_\ell$.

Now, we apply Lemma~\ref{lem:explicit} to $g=g_{\ell,N}-\gamma_{N}h_{\ell,N+1}$, with $M=2$ and successively for $s=2,p>N+1$, $s=1,p=2$.
Then, we get the existence of $\widetilde{h}_{\ell,N+1}\in \left(\bigcap_{p>N+1}H^{2+1/p}_p(\R^{N+1},\R)\right)\cap H^{1+1/2}(\R^{N+1},\R)\subset \left(\bigcap_{p>N}H^{2}_p(\R^{N+1},\R)\right)\cap H^{1}(\R^{N+1},\R)$ such that $\gamma_N \widetilde{h}_{\ell,N+1}=g_{\ell,N}-\gamma_{N}h_{\ell,N+1}$, which satisfies~
\begin{itemize}
\item $\|\widetilde{h}_{\ell,N+1}\|_{H^{2}_p(\R^N,\R)}\leq 2K_2 D_{\ell,p}^N$ with $p>N+1$. In view of $g_{\ell,N+1}=h_{\ell,N+1}+\widetilde{h}_{\ell,N+1}$ and $D_{\ell,p}\geq \max(C_\ell,2K_2+1)$, it implies $\|\widetilde{h}_{\ell,N+1}\|_{H^{2}_p(\R^N,\R)}\leq \|h_{\ell,N+1}\|_{H^{2}_p(\R^N,\R)}+\|\widetilde{h}_{\ell,N+1}\|_{H^{2}_p(\R^N,\R)}\leq C_\ell+2K_2 D_{\ell,p}^N\leq D_{\ell,p}^{N+1}$ which means that~\eqref{e:cv:induction1} is satisfied for $n=N+1$.
\item $\|\widetilde{h}_{\ell,N+1}\|_{H^1(\R^{n+1},\R)} \leq K_2\,\|g_{\ell,N}-\gamma_{N}h_{\ell,N+1}\|_{H^{1/2}(\R^n,\R)}$. By~\eqref{e:cv:induction2} we can replaced $C$ by $\max(C,K_2)$ and deduce
\begin{align}
  \|\widetilde{h}_{\ell,N+1}\|_{H^1(\R^{N+1},\R)} &\leq C\,\|g_{\ell,N}-\gamma_{N}h_{\ell,N+1}\|_{H^{1/2}(\R^n,\R)}\nonumber \\
  &\leq C^N \|g_{\ell,0}-\gamma_{N}h_{\ell,1}\|_{H^{1/2}(\R^n,\R)}\leq C^N/\ell.\label{e:ineq:induc}
\end{align}
which means that~\eqref{e:cv:induction2} holds for $n=N+1$.
\end{itemize}
We then define a sequence $(g_{\ell,N})$ which satisfies the two induction assumptions for $n=N+1$.

\noindent{\it Step~3: Proof of Properties (A)--(D) for $g_{\ell,N+1}=h_{\ell,N+1}+\widetilde{h}_{\ell,N+1}$}.
Property (A) is obvious by definition of $(g_{\ell,0})$. By construction, we have that, for any $\ell$, $\widetilde{h}_{\ell,N+1}$ is a symmetric function belonging to $\bigcap_{p>N+1}H^{2}_{p}(\R^{n+1},\R)\subset \mathcal{C}^1(\R^{n+1},\R)$ (see Lemma~\ref{lem:emb-C1}), and so is $g_{\ell,N+1}$. Further, applying once more Lemma~\ref{lem:emb-C1} with $p=p_n>3/2n$ and property~\eqref{e:cv:induction1} of the sequence $(g_{\ell,n})$ we get
\[
\|g_{\ell,n}\|_{L^\infty(\R^n,\R)}+\|\nabla g_{\ell,n}\|_{L^\infty(\R^n,\R)}\leq \widetilde{D}_\ell^n
\]
with $\widetilde{D}_\ell=\Lambda_{1/2}D_{\ell,p_n}$. Hence  (B) is satisfied. The relation $\gamma_{N}g_{\ell,N+1}=g_{\ell,N}$ directly ensues from their definitions and yields Property (D).  Since $g_{\ell,N+1}-\Phi_{N+1}=(h_{\ell,N+1}-\Phi_{N+1})+\widetilde{h}_{\ell,N+1}$, combining \eqref{e:density-hnl} and~\eqref{e:ineq:induc} implies that $(C)$ is satisfied for $n=N+1$.
\end{proof}

\bigskip

\noindent {\Large \bf Acknowledgements}

\medskip

\noindent The authors would like to thank Giovanni Conforti, Nicolas Privault, Mathias Rafler and Sylvie Roelly for fruitful discussions and for providing interesting references. The authors are also grateful to the associate editor and the two anonymous referees who helped to improve a previous version of the manuscript.

\bibliographystyle{plainnat} 
\bibliography{stein}

\end{document}